\documentclass[11pt]{amsart}
\usepackage{latexsym, mathrsfs, color, tikz, multirow,bbm,mathtools,amsthm}
\usepackage{graphics}
\usepackage{subcaption}
\input{xy}
\xyoption{all}

\usepackage[colorlinks=true, pdfstartview=FitV,
 linkcolor=black,citecolor=black,urlcolor=black]{hyperref}
 
\usepackage{tikz}
\usepackage{tkz-graph}
\usetikzlibrary{automata,arrows}
\usetikzlibrary{shapes.geometric}
\usetikzlibrary{patterns}
\usepackage{tkz-euclide}

\numberwithin{equation}{section}
\theoremstyle{plain}
\newtheorem{Thm}{Theorem}[section]
\newtheorem{Prop}[Thm]{Proposition}

\newtheorem{Lem}[Thm]{Lemma}
\newtheorem{Conj}[Thm]{Conjecture}

\theoremstyle{definition}
\newtheorem{Def}[Thm]{Definition}
\newtheorem{Exa}[Thm]{Example}
\newtheorem{Rmk}[Thm]{Remark}
\newtheorem{LDef}{Definition $\Lambda$-\hspace{-.4em}}
\newtheorem{LExa}{Example $\Lambda$-\hspace{-.4em}}

\newenvironment{red}{\relax\color{red}}{\relax}
\newenvironment{blue}{\relax\color{blue}}{\hspace*{.5ex}\relax}

\newcommand{\ber}{\begin{red}}
\newcommand{\er}{\end{red}}
\newcommand{\beb}{\begin{blue}}
\newcommand{\eb}{\end{blue}}

\newcommand{\bw}{{\boldsymbol{w}}}
\newcommand{\bv}{{\boldsymbol{v}}}

\begin{document}
\title{Geometric description of  $C$-vectors and real L\"osungen}
\author[K.-H. Lee]{Kyu-Hwan Lee$^{\star}$}
\thanks{$^{\star}$This work was partially supported by a grant from the Simons Foundation (\#712100).}
\address{Department of
Mathematics, University of Connecticut, Storrs, CT 06269, U.S.A.}
\email{khlee@math.uconn.edu}

\author[K. Lee]{Kyungyong Lee$^{\dagger}$}
\thanks{$^{\dagger}$This work was partially supported by NSF grant DMS 1800207, the University of Alabama, and Korea Institute for Advanced Study.}
\address{Department of Mathematics, University of Alabama,
Tuscaloosa, AL 35487, U.S.A. 
and Korea Institute for Advanced Study, Seoul 02455, Republic of Korea}
\email{klee94@ua.edu; klee1@kias.re.kr}

\author[M.~R. Mills]{Matthew R. Mills$^{\diamond}$}
\thanks{$^{\diamond}$This material is based upon work supported by the National Science Foundation under Award No.~1803521 and Michigan State University.}
\address{Department of Mathematics, Michigan State University,
East Lansing, MI 48824, U.S.A.}
\email{millsm12@msu.edu}

\begin{abstract}
We introduce real L\"osungen as an analogue of real roots. For each mutation sequence of an arbitrary skew-symmetrizable matrix, we define a family of reflections along with associated vectors which are real L\"osungen and a set of curves on a Riemann surface. The matrix consisting of these vectors is called {\em $L$-matrix}. We explain how the $L$-matrix naturally arises in connection with the $C$-matrix. Then we conjecture that the $L$-matrix depends (up to signs of row vectors) only on the seed, and that the curves can be drawn without self-intersections, providing a new combinatorial/geometric description of $c$-vectors. 
\end{abstract}

\maketitle

\section{Introduction}

Let $Q$ be a quiver  with $n$ vertices and no oriented cycles of length $\le 2$.  The most basic invariant of a representation of $Q$ is its dimension vector. By Kac's Theorem \cite{Kac-1}, the dimension vectors of indecomposable representations of $Q$ are positive roots of the Kac--Moody algebra $\mathfrak g_Q$ associated to the quiver $Q$.  

When $Q$ is acyclic, a representation $M$ of $Q$ is called {\em rigid} if $\mathrm{Ext}^1(M, M)=0$, and the dimension vectors of indecomposable rigid representations are called {\em real Schur roots} as they are indeed real roots of $\mathfrak g_Q$. In the category of representations of $Q$, rigid objects are foundational. Therefore an explicit  description of real Schur roots is essential for the study of the category, and there have been various results related to description of real Schur roots of an acyclic quiver (\cite{BDSW,HK,IS,S,Se,ST}).

In a previous paper \cite{LL}, we conjectured  a correspondence   between real Schur roots of an acyclic quiver  and  non-self-crossing curves on a marked Riemann surface and hence proposed a new combinatorial/geometric description.
Recently, Felikson and Tumarkin \cite{FT}  proved our conjecture for all $2$-complete acyclic quivers. (An acyclic quiver is called  {\it $2$-complete} if it has multiple edges between any pair of vertices.)

Now, when $Q$ is general, it is natural to consider the {\em $c$-vectors} of $Q$ as dimension vectors of rigid objects. Indeed, when $Q$ is acyclic, the set of positive $c$-vectors is identical with  the set of real Schur roots 
\cite{NC}. For an arbitrary quiver $Q$,
a positive $c$-vector is the dimension vector of a rigid indecomposable representation of a quotient of the completed path algebra. This quotient was introduced by Derksen, Weyman and Zelevinksy \cite{DWZ}, and is called a {\em Jacobian algebra}. Thus $c$-vectors naturally generalize real Schur roots in this sense, though they are not necessarily real roots of the corresponding Kac--Moody algebra. 

Originally, $c$-vectors (and $C$-matrices) were defined in the theory of cluster algebras \cite{FZ}, and together with their companions, $g$-vectors (and $G$-matrices), played  fundamental roles in the study of cluster algebras (for instance, see \cite{DWZ, GHKK, MG, NZ, Pl}). As a cluster algebra is defined not only for a skew-symmetric matrix (i.e. a quiver) but also for an arbitrary skew-symmetrizable matrix, one can ask: \begin{quote} \em Can we have a combinatorial/geometric description of  the $c$-vectors (and $C$-matrices) of a cluster algebra associated with an arbitrary skew-symmetrizable matrix? 
\end{quote}
In this paper, we propose a conjectural, combinatorial/geometric model for $C$-matrices associated to an arbitrary skew-symmetrizable matrix, which extends our model from the acyclic case \cite{LL}.

For this purpose, we introduce the notion of {\em real L\"osungen} as an analogue of real roots, 
 and define a family of reflections along with associated vectors which are real L\"osungen for each mutation sequence of an arbitrary skew-symmetrizable matrix. The matrix consisting of these real L\"osungen is called {\em $L$-matrix}. We show that the $L$-matrix comes from certain leading terms when the $C$-matrix is presented using reflections. We conjecture that the $L$-matrices (up to signs of row vectors) depend
only on seeds, i.e., do not
depend on mutation sequences leading to the same seed. We believe
that understanding these new matrices is a key to generalizing Coxeter groups and their quotients
arising from cluster algebras, in particular, generalizing Felikson--Tumarkin's result 
\cite{FT2}.

\medskip

When a  skew-symmetrizable matrix is acyclic,  it is natural to consider the corresponding symmetrizable generalized Cartan matrix. For a general skew-symmetrizable matrix, we consider {\em generalized intersection matrices} (GIMs)\footnote{Some authors  call them {\em quasi-Cartan matrices}. For example, see \cite{BR}.} introduced by  Slodowy \cite{Slo,Slo-1}. 
A GIM is a square matrix $A=[a_{ij}]$ with integral entries such that
\begin{enumerate}
\item for diagonal entries, $a_{ii}=2$;
\item $a_{ij}>0$ if and only if $a_{ji}>0$;
\item $a_{ij}<0$ if and only if $a_{ji}<0$.
\end{enumerate}

Since we are more interested in cluster algebras associated with skew-symmetrizable matrices, we restrict ourselves to the  class of symmetrizable GIMs.
This class contains the collection of all symmetrizable generalized Cartan matrices as a special subclass.  

Let $\mathcal A$ be the (unital) $\mathbb Z$-algebra generated by $s_i, e_i$, $i=1,2, \dots, n$, subject to the following relations:
$$ s_i^2=1, \quad \sum_{i=1}^n e_i =1, \quad s_ie_i = -e_i, \quad e_is_j=  \begin{cases} s_i+e_i-1 &\text{if } i =j, \\ e_i &\text{if } i \neq j, \end{cases}  \quad e_ie_j=  \begin{cases} e_i &\text{if } i =j, \\ 0 & \text{if } i \neq j. \end{cases}$$
Let $\mathcal W$ be the subgroup of the units of $\mathcal{A}$ generated by $s_i$, $i=1, \dots, n$. Note that $\mathcal W$ is (isomorphic to) the universal Coxeter group. Thus the algebra $\mathcal A$ can be considered as the algebra generated by the reflections and projections of the universal Coxeter group. Keeping computations at the level of $\mathcal A$ will reveal some important features of mutations.

\begin{Def} \label{def-gim}
Let $A=[a_{ij}]$ be an $n\times n$ symmetrizable GIM, and $D=\mathrm{diag}(d_1, \dots , d_n)$ be the {\em symmetrizer}, i.e. the diagonal matrix such that $d_i \in \mathbb Z_{>0}$, $\gcd(d_1, \dots, d_n)=1$ and $AD$ is symmetric. Let $\Gamma = \sum_{i=1}^n \mathbb{Z}\alpha_i$ be the lattice generated by the formal symbols $\alpha_1,\cdots,\alpha_n$. 
\begin{enumerate}

\item An element $\gamma=\sum m_i\alpha_i \in \Gamma$ is called a {\em L\"{o}sung} if \begin{equation} \label{eqn-quad} \sum_{1\leq i,j\leq n} d_ja_{ij}  m_im_j =2 d_k \quad \text{ for some }k=1, \dots , n .\end{equation} 
A L\"{o}sung is {\em positive} if $m_i\geq 0$ for all $i$. Each $\alpha_i$ is called a {\em simple L\"{o}sung}.

\item Define a representation $\pi:\mathcal A \rightarrow \mathrm{End}(\Gamma)$ by
\[ \pi(s_i)(\alpha_j) = \alpha_j - a_{ji} \alpha_i \quad \text{ and } \quad \pi(e_i)(\alpha_j) =\delta_{ij} \alpha_i , \quad i,j=1, \dots ,n. \] We suppress $\pi$ when we write the action of an element of $\mathcal A$ on $\Gamma$.
A L\"{o}sung $\gamma$ is {\em real} if $\gamma=s_{i_1}s_{i_2} \cdots s_{i_k}(\alpha_i)$ for some $i=1, \dots, n$ and $k \ge 0$.

\end{enumerate}
\end{Def}

\begin{Rmk}
When $A$ is symmetric, a L\"osung is also called a {\em root} in some literature. For example, see \cite{ASS,Ri}. When $A$ is a generalized Cartan matrix of finite, affine or hyperbolic type, this terminology does not bring any confusion with a {root}\footnote{Historically, when Killing investigated the structure of a finite dimensional simple Lie algebra $L$ with Cartan subalgebra $\mathfrak h$, the roots of the characteristic polynomial $\det(\operatorname {ad} _{L}x-t)$, $x \in \mathfrak h$, were called the {roots} \cite{Bour}.} of the root system associated with $A$ because a L\"osung is a root of the root system \cite[Proposition 5.10]{Kac}. However, in general, a L\"osung is not a root of the root system. See \cite[p.11]{NF} for the case when $A$ is of type $E_{11}$. In order to avoid possible confusion, we introduce the term L\"osung to distinguish it from a root of a root system.  
 
Nevertheless, if $A$ is a generalized Cartan matrix, real L\"{o}sungen are the same as real roots of the Kac--Moody algebra associated with $A$. We expect that, for each symmetrizable GIM, there may exist a Lie algebra for which real roots can be defined and are compatible with real L\"{o}sungen, but we do not yet know which Lie algebra would be adequate. Some related works can be found in \cite{BR,BKL, BZ,BM, SY,Slo-1,Slo,XH}.
\end{Rmk}

Fix an $n\times n$ skew-symmetrizable matrix $B=[b_{ij}]$ and let $D=\mathrm{diag}(d_1, \dots, d_n)$ be its symmetrizer such that $BD$ is skew-symmetric, $d_i \in \mathbb Z_{>0}$ and $\gcd(d_1, \dots, d_n)=1$.  Consider the $n\times 2n$ matrix $\begin{bmatrix}B&I\end{bmatrix}$.  After a sequence $\bw$ of mutations, we obtain $\begin{bmatrix} B^\bw &C^\bw \end{bmatrix}$. The matrix $C^\bw$ is called the {\em $C$-matrix} and its row vectors the {\em $c$-vectors}. 
Write their entries as 
\begin{equation} \label{eqn-bcbc} B^\bw= \begin{bmatrix} b_{ij}^\bw \end{bmatrix}, \qquad
C^\bw= \begin{bmatrix} c_{ij}^\bw \end{bmatrix} = \begin{bmatrix} c_1^\bw \\ \vdots \\ c_n^\bw \end{bmatrix},\end{equation}
where $c_i^\bw$ are the $c$-vectors.
For a mutation sequence $\bw=[i_i, i_2, \dots , i_\ell]$, $i_j \in \{1,2,\dots, n\}$, we define $\bw[k]:= [i_i, i_2, \dots , i_\ell,k]$.
\begin{Def} \label{def-r}
For each  mutation sequence $\bw$, define $r_i^\bw \in \mathcal W \subset \mathcal A$ inductively  with the initial elements $r_i = s_i$,  $i=1, \dots, n$, as follows: 
\begin{equation} \label{def-sx_i-1} r_i^{\bw [k]}= \begin{cases} r_k^\bw r_i^\bw  r_k^\bw & \text{ if } \ b_{ik}^\bw c_{k}^\bw>0, \\ r_i^\bw & \text{ otherwise.} \end{cases} \end{equation}
Clearly, each $r_i^\bw$ is written in the form
\[ r_i^\bw = g_i^\bw s_i {(g_i^\bw)}^{-1}, \quad g_i^\bw \in \mathcal W, \quad i=1, \dots , n.\]
\end{Def}

This construction has been used in the literature including \cite{BMa,FT2,FT,ST} when the associated GIM is a Cartan matrix.

\begin{Def}\label{def-ell}
Fix a GIM $A$, and define
\[   l_i^\bw = g_i^\bw (\alpha_i), \qquad i=1, \dots , n. \]
Then the {\em $L$-matrix} $L^{\bw}$ associated to $A$ is defined to be the $n \times n$ matrix whose $i^\text{th}$ row is $l_i^\bw$ for $i= 1, \dots, n$, i.e., \[ L^{\bw} =  \begin{bmatrix} l_1^{\bw}  \\ \vdots \\ l_n^{\bw} \end{bmatrix},\] and the vectors $l_i^{\bw}$ are called the {\em $l$-vectors of $A$}. 

Note that the $L$-matrix and $l$-vectors associated to a GIM $A$ implicitly depend on the representation $\pi$ which is suppressed from the notation. When multiple GIMs are being discussed we will use the notation $l_i^{A,\bw}$ to distinguish between different sets of $l$-vectors. 
\end{Def}

When we fix a GIM, we will always choose a linear ordering $\prec$ on $\{ 1,2, \dots , n\}$ and define the associated GIM $A=[a_{ij}]$ by
\begin{equation} 
a_{ij}= \begin{cases}  b_{ij}  & \text{ if } i \prec j , \\
2 & \text{ if } i =j, \\ -b_{ij} & \text{ if } i \succ j . \end{cases}\ 
\label{eqn-gim}
\end{equation}
An ordering $\prec$ provides a certain way for us to regard the skew-symmetrizable matrix $B$ as acyclic even when it is not. 

As our geometric model, we consider a Riemann surface and admissible curves (Definition \ref{def-adm}), and define a map from the set of admissible curves to the set of monomials in $s_i$'s in $\mathcal W$ (Definition \ref{def-many}).
The first conjecture below extends our conjecture in \cite{LL} from acyclic quivers to skew-symmetrizable matrices. The second conjecture claims that we can choose a GIM $A$ to obtain a set of reflections that only depend on the seed.

\begin{Conj}\label{vague_conj-2}
Fix an ordering $\prec$ on $\{1,2, \dots , n \}$ so that a GIM $A$ is determined. Then for any mutation sequence $\bw$, there exist non-self-intersecting admissible curves $\eta_i^\bw$ such that $\pi(r_i^\bw) =  \pi \left (s(\eta_i^\bw)\right ),$
where $s(\eta_i^\bw)$ are the monomials in $\mathcal W$ associated to $\eta_i^\bw$ for $i=1,2,\dots , n$.
\end{Conj}

\begin{Conj}\label{vague_conj-3}
For any skew-symmetrizable matrix $B$, there exists a linear ordering $\prec$ and its associated GIM $A$ such that if $\bw$ and $\bv$ are two mutation sequences with $C^\bw=C^\bv$ then
$\pi(r_i^\bw) =\pi(r_i^\bv), \, i=1,\dots,n.$
\end{Conj}

For any acyclic skew-symmetrizable matrix, choosing a linear ordering where $i \prec j$ if and only if $b_{ij}<0$ yields a GIM that is a Cartan matrix by \eqref{eqn-gim}. In this case, Conjecture~\ref{vague_conj-3} has been proven in \cite{ST} using some results from categorification of cluster algebras.

\medskip

As the main result of this paper, we show that the reflections $r_i^\bw$ naturally arise in connection with the $C$-matrix. It also justifies potential importance of the matrix $L^\bw$. 
The key idea is to maintain that we should have a ``root system'' for each mutation sequence $\bw$ as in the acyclic case. More precisely, we 
choose a linear ordering $\prec$ and its associated GIM, and inductively define  an $n$-tuple of elements $s_i^\bw \in \mathcal A$ and an $n$-tuple of vectors $\lambda_i^\bw \in \mathbb Z^n \,( \cong \Gamma)$, $i=1,2, \dots, n$,
 so that the following formulae hold:
\begin{align}\label{formulae} 
s_i^\bw(\lambda_j^\bw)& = \begin{cases} \lambda_j^\bw +b_{ji}^\bw \lambda_i^\bw & \text{ if } i \prec j, \\ -\lambda_j^\bw & \text{ if } i =j, \\\lambda_j^\bw -b_{ji}^\bw \lambda_i^\bw & \text{ if } i \succ j ,  \end{cases} 
\end{align}
where $B^\bw=[b_{ij}^\bw]$.
We denote by $\Lambda^\bw$ the matrix whose rows are $\lambda_i^\bw$. 
\begin{Thm}\label{vague_conj-1}
Fix a linear ordering $\prec$ on $\{1,2, \dots, n \}$ to obtain its associated GIM $A$. Then, for each mutation sequence $\bw$, we have
\[ \Lambda^\bw =C^\bw.\] Moreover, \[ s_i^\bw \equiv r_i^\bw \quad (\mathrm{mod}\ 2 \mathcal A), \quad i=1,2,\dots , n.\]  
\end{Thm}

As one can see from the flow chart in Table \ref{tab-1}, the definitions of $s_i^\bw$ and $\lambda_i^\bw$ are somewhat convoluted and heavily depend on $\prec$. Nevertheless, in the end, we obtain $C^\bw$ and $r_i^\bw$ which do not depend on $\prec$. Moreover, this process reveals that $r_i^\bw$ are certain leading terms in $s_i^\bw$. 
Since $s_i^\bw$ are related to $\lambda_i^\bw$ and $r_i^\bw$ to $l_i^\bw$, the $l$-vectors $l_i^\bw$ can be considered as ``leading terms'' of the $c$-vectors $c_i^\bw(=\lambda_i^\bw)$. What  
 Conjectures \ref{vague_conj-2} and \ref{vague_conj-3} claim is that these leading terms carry essential information. 
 
\medskip

To illustrate Theorem \ref{vague_conj-1}, we present Example~\ref{exa-1} below. Conjecture~\ref{vague_conj-2} will be checked for this example in Example~\ref{exa-pr} after an admissible curve is defined. Conjecture~\ref{vague_conj-3} is trivially satisfied for this matrix since its exchange graph is a tree (see \cite{Se2}) and thus $C^\bv = C^\bw$ does not occur (unless $\bv$ and $\bw$ differ only by repeated mutations $[i,i]$ at the same index). A non-trivial example of Conjecture~\ref{vague_conj-3} is given in Example~\ref{exa-pi-wv}.

\begin{Exa} \label{exa-1}

Consider the skew-symmetrizable matrix $B= {\small \begin{bmatrix}
0&3&-3\\-2&0&2\\2&-2&0
\end{bmatrix}}$ with the symmetrizer $D=\mathrm{diag}(3,2,2)$, and the sequence of consecutive mutations at indices $2,3,2,1,2$:
\begin{center}
 {\small $\begin{bmatrix}
B & I \end{bmatrix}$} \qquad $\xrightarrow{\quad [2,3,2,1,2] \quad }$ \qquad 
{\small $\begin{bmatrix}
0&-3&9&5&18&15\\2&0&-4&-2&-7&-6\\-6&4&0&0&-2&-1
\end{bmatrix}$}
\end{center}
Thus we have obtained three $c$-vectors $(5,18,15)$, $(-2,-7,-6)$ and $(0,-2,-1)$. 

We take the linear ordering $1 \succ 2 \succ 3$. Then its GIM $A$ and the symmetrized matrix $AD$ are as follows: 
\[ \small A= \begin{bmatrix} 2&-3&3\\-2&2&-2\\2&-2&2 \end{bmatrix},  \qquad AD= \begin{bmatrix} 6&-6&6\\-6&4&-4\\6&-4&4 \end{bmatrix}. \]
In accordance with \eqref{eqn-quad}, define a quadratic form by
\[ q(x,y,z)=6x^2+4y^2+4z^2-12xy-8yz+12zx .\]
Then we have
\[ q(5,18,15)=6, \quad  q(-2,-7,-6)=4, \quad q(0,-2,-1)=4.  \]
Thus all three $c$-vectors are L{\"o}sungen for $A$.

From Definition~\ref{def-r}, we obtain 
\[ r_1^\bv =s_3 s_2 s_1 s_2 s_3 s_2 s_3 s_2 s_1 s_2 s_3 s_2 s_3 s_2 s_1 s_2 s_3, \quad  
r_2^\bv = s_3 s_2 s_1 s_2 s_3 s_2 s_3 s_2 s_1 s_2 s_3, \quad 
r_3^\bv = s_2 s_3 s_2,\]
where $\bv$ is the mutation sequence $[2,3,2,1,2]$.
For the GIM $A$, Definition~\ref{def-ell} gives rise to the $l$-vectors 
\begin{align*} l_1^{\bv} & = s_3 s_2 s_1 s_2 s_3 s_2 s_3 s_2 (\alpha_1)=(5,18,15), \\ l_2^{\bv} & = s_3 s_2 s_1 s_2 s_3 (\alpha_2)=(2,7,6), \quad l_3^{\bv}=  s_2  (\alpha_3)=(0,2,1). \end{align*}
On the other hand, following the definitions in Section \ref{sec-conj}, we obtain similar results for the $\lambda_i^\bw$. In particular,
\begin{align*} \lambda_1^\bv & = s_3 s_2 s_1 s_2 s_3 s_2 s_3 s_2 (\alpha_1)=(5,18,15), \\ \lambda_2^\bv & = -s_3 s_2 s_1 s_2 s_3 (\alpha_2)=(-2,-7,-6), \quad \lambda_3^\bv= - s_2  (\alpha_3)=(0,-2,-1). \end{align*}
Thus the matrix {\small $\Lambda^\bv=\begin{bmatrix}
5&18&15\\-2&-7&-6\\0&-2&-1
\end{bmatrix}$} equals the $C$-matrix. 

However, $l$-vectors will not always be equal to positive $c$-vectors. Indeed, they need not even be sign-coherent. For the choice of GIM $A' =  \begin{bmatrix} 2&3&-3\\2&2&2\\-2&2&2 \end{bmatrix}$ we see that \[ l_1^{A',\bv} =(149, -462 ,1341),\quad l_2^{A',\bv} =(-10,31 ,-90),\quad l_2^{A',\bv} =(0,-2,1).\]
\end{Exa}

\subsection{Organization of the paper} In Section \ref{sec-conj}, precise definitions will be made for the objects appeared in this introduction, and Conjectures \ref{vague_conj-2} and \ref{vague_conj-3} will be presented in a more refined way, and other examples will be given. In Section \ref{sec-results} the elements $s_i^\bw \in \mathcal A$ and the vectors $\lambda_i^\bw$ will be defined with a running example, and Theorem \ref{vague_conj-1} will be stated more precisely.  
In Section \ref{sec-proof-c1c2}, Theorem \ref{vague_conj-1} will be proven through induction. The main induction step consists of six different cases, each of which has a few subcases.

\subsection*{Acknowledgments} We are very grateful to Pavel Tumarkin, Ahmet Seven and anonymous referees for correspondences and comments, which substantially improved the exposition of this paper.

\section{Conjectures} \label{sec-conj}

In this section, we present our conjectures in a more precise way after making necessary definitions. 

\medskip

For a nonzero vector $c=( c_1 , \dots , c_n) \in \mathbb Z^n$, we define $c >0$ if all $c_i$ are non-negative, and  $c <0$ if all $c_i$ are non-positive. This induces a partial ordering $<$ on $\mathbb Z^n$. Define $|c|=( |c_1|, \dots,  |c_n| )$.

Assume that $M=[m_{ij}]$ is an $n \times 2n$ matrix of integers. Let $\mathcal I:= \{ 1, 2, \dots, n \}$ be the set of indices. For $\bw=[i_i, i_2, \dots , i_\ell]$, $i_j \in \mathcal I$, we define the matrix $M^\bw=[m_{ij}^\bw]$ inductively: the initial matrix is $M$ for $\bw=[\,]$, and assuming we have $M^\bw$, define the matrix $M^{\bw[k]}=[m_{ij}^{\bw[k]}]$ for $k \in \mathcal I$ with $\bw[k]:=[i_i, i_2, \dots , i_\ell, k]$
by  
\begin{equation} \label{eqn-mmuu} m_{ij}^{\bw[k]} = \begin{cases} -m_{ij}^\bw & \text{if  $i=k$ or $j=k$}, \\ m_{ij}^\bw + \mathrm{sgn}(m_{ik}^\bw) \, \max(m_{ik}^\bw m_{kj}^\bw,0) & \text{otherwise}, \end{cases}
\end{equation}
where  $\mathrm{sgn}(a) \in \{1, 0, -1\}$ is the signature of $a$. The matrix $M^{\bw[k]}$ is called the {\em mutation of $M^\bw$} at the index $k$.

Let $B=[b_{ij}]$ be an $n\times n$ skew-symmetrizable matrix  and $D=\mathrm{diag}(d_1, \dots, d_n)$ be its symmetrizer such that $BD$ is symmetric, $d_i \in \mathbb Z_{>0}$ and $\gcd(d_1, \dots, d_n)=1$.  Consider the $n\times 2n$ matrix $\begin{bmatrix}B&I\end{bmatrix}$ and
 a mutation sequence $\bw=[i_1, \dots ,i_k]$. After the mutations at the indices $i_1, \dots , i_k$ consecutively, we obtain $\begin{bmatrix}B^\bw&C^\bw\end{bmatrix}$. Write their entries as in \eqref{eqn-bcbc}. 
It is well-known that the $c$-vector $c_i^\bw$ is non-zero for each $i$, and either $c_i^\bw >0$ or $c_i^\bw <0$  due to sign coherence of $c$-vectors  (\cite{DWZ-1, GHKK}).

Choose a linear ordering $\prec$ on the set $\mathcal I$, and define a GIM $A=[a_{ij}]$ by \eqref{eqn-gim}.
From Definition \ref{def-gim}, we have L\"osungen associated with $A$.
Set $\lambda_1=(1,0,\dots , 0), \ \lambda_2=(0,1,0, \dots, 0), \dots, \ \lambda_n=(0, \dots, 0, 1)$ to be a basis of $\mathbb Z^n$.  
Recall that we have defined the algebra $\mathcal A$ in the introduction.
Define a representation $\pi:\mathcal A \rightarrow \mathrm{End}(\mathbb Z^n)$ by 
\begin{equation} \label{s_i-[]}
\pi(s_i)(\lambda_j) = \lambda_j-a_{ji} \lambda_i \quad \text{ and } \quad \pi(e_i)(\lambda_j)  = \delta_{ij}\lambda_i \qquad \text{ for } i,j \in \mathcal I, 
\end{equation} and by extending it through linearity, 
where $\delta_{ij}$ is the Kronecker delta. We will suppress $\pi$ when we write the action of an element of $\mathcal A$ on $\mathbb Z^n$. 
As before, denote by $\mathcal W$  the subgroup of the units of $\mathcal{A}$ generated by $s_i$, $i=1, \dots, n$.

\medskip

To introduce our geometric model\footnote{An alternative geometric model can be found in \cite{FT}.} for $c$-vectors, we need a Riemann surface  equipped with $n$ labeled curves as below.
Let $P_1$ and $P_2$ be two identical copies of a regular $n$-gon. 
For $\sigma \in S_n$, label the edges of each of the two  $n$-gons
 by $T_{\sigma (1)}, T_{\sigma (2)}, \dots , T_{\sigma (n)}$ counter-clockwise.

 On $P_i$ $(i=1,2)$, let $L_i$ be the line segment from the center of $P_i$ to the common endpoint of  $T_{\sigma (1)}$ and $T_{\sigma (2)}$. Later,  these line segments will only be used  to designate the end points of admissible curves and will not be used elsewhere.   Fix the orientation of every edge of $P_1$ (resp.  $P_2$) to be 
 counter-clockwise (resp. clockwise) as in the following picture. 
 \begin{center}
 \begin{tikzpicture}[scale=0.5]
\node at (2.5,-2.7){\tiny{$T_{\sigma (n-1)}$}};
\node at (1.5,2.6){\tiny{$T_{\sigma (1)}$}};
\node at (4,0){\tiny{$T_{\sigma (n)}$}};
\node at (-2.0,-2.7){\tiny{$T_{\sigma (n-2)}$}};
\node at (-2.2,2.5){\tiny{$T_{\sigma (2)}$}};
\node at (-3.2,0){\vdots};
\draw (0,0) +(30:3cm) -- +(90:3cm) -- +(150:3cm) -- +(210:3cm) --
+(270:3cm) -- +(330:3cm) -- cycle;
\draw [thick] (2.4,-0.2) -- (2.6,0)--(2.8,-0.2);
\draw [thick] (1.4,1.95) -- (1.3,2.25)--(1.6,2.25);  
\draw [thick] (-1.0,2.2) -- (-1.3,2.2)--(-1.2,2.5);  
\draw [thick] (-2.4,0.2) -- (-2.6,0)--(-2.8,0.2);   
\draw [thick] (1.0,-2.2) -- (1.3,-2.2)--(1.2,-2.5);  
\draw [thick] (-1.4,-1.95) -- (-1.3,-2.25)--(-1.6,-2.25);  
\draw [thick] (0,0)--(0,3);  
\node at (0.6,1.2){\tiny{$L_1$}};
\end{tikzpicture}
\begin{tikzpicture}[scale=0.5]
\node at (-0.6,-1.2){\tiny{$L_2$}};
\draw [thick] (0,0)--(0,-3);  
\node at (2.8,-2.4){\tiny{$T_{\sigma (2)}$}};
\node at (1.5,2.9){\tiny{$T_{\sigma (n-2)}$}};
\node at (3.3,0){\vdots};
\node at (-1.9,-2.7){\tiny{$T_{\sigma (1)}$}};
\node at (-2.0,2.7){\tiny{$T_{\sigma (n-1)}$}};
\draw (0,0) +(30:3cm) -- +(90:3cm) -- +(150:3cm) -- +(210:3cm) --
+(270:3cm) -- +(330:3cm) -- cycle;
\draw [thick] (-2.4,-0.2) -- (-2.6,0)--(-2.8,-0.2);
\draw [thick] (-1.4,1.95) -- (-1.3,2.25)--(-1.6,2.25);  
\draw [thick] (1.0,2.2) -- (1.3,2.2)--(1.2,2.5);  
\draw [thick] (2.4,0.2) -- (2.6,0)--(2.8,0.2);   
\draw [thick] (-1.0,-2.2) -- (-1.3,-2.2)--(-1.2,-2.5);  
\draw [thick] (1.4,-1.95) -- (1.3,-2.25)--(1.6,-2.25);  
\end{tikzpicture}
 \end{center}
 
Let $\Sigma_\sigma$ be the Riemann surface of genus $\lfloor \frac{n-1}{2}\rfloor$
obtained by gluing together the two $n$-gons with all the edges of the same label identified according 
to their orientations.  The edges of the $n$-gons become $n$ different curves in $\Sigma_\sigma$. If $n$ is odd, all the vertices of the two $n$-gons 
are identified to become one point in $\Sigma_\sigma$ and the curves obtained from the edges become loops. If $n$ is even, two distinct
 vertices are shared by all curves. Let ${T}={T}_1\cup\cdots{T}_n\subset \Sigma_\sigma$, and $V$ be the set of the vertex (or vertices) on ${T}$.  

Let $\mathfrak W$ be the universal Coxeter group of rank $n$, which is by definition isomorphic to the free product of $n$-copies of $\mathbb Z/2 \mathbb Z$, and let $\mathfrak R$ be the set of reflections in $\mathfrak W$. We will denote an element of $\mathfrak W$ as a word from the alphabet $\mathcal I=\{1,2,...,n\}$.  In particular, an element $\mathfrak v$ of $\mathfrak R$ can be written as  $\mathfrak v=i_1i_2 \cdots i_k$ such that $k$ is an odd integer and $i_{j}=i_{k+1-j}$ for all $j=1,2, \dots , k$.

\begin{Def} \label{def-adm} 
An \emph{admissible curve} is a continuous function $\eta:[0,1]\longrightarrow \Sigma_\sigma$ such that

1) $\eta(x)\in V$ if and only if  $x\in\{0,1\}$;

2) there exists $\epsilon>0$ such that $\eta([0,\epsilon])\subset L_1$ and $\eta([1-\epsilon,1])\subset L_2$;

3) if $\eta(x)\in {T}\setminus V$ then $\eta([x-\epsilon,x+\epsilon])$ meets ${T}$ transversally for sufficiently small $\epsilon>0$;

4) $\upsilon(\eta)\in \mathfrak R$, where $\upsilon(\eta):={i_1}\cdots {i_k} \in \mathfrak W$ is given by 
$$\{x\in(0,1) \ : \ \eta(x)\in {T}\}=\{x_1<\cdots<x_k\}\quad \text{ and }\quad \eta(x_\ell)\in T_{i_\ell}\text{ for }\ell\in\{1,...,k\}.$$ 

\end{Def}

We consider curves up to isotopy. When $i_p=i_{p+1}$, $1 \le p \le k-1$, for $\upsilon(\eta)={i_1}\cdots {i_k}$, the curve  $\eta$  is isotopic  to a curve $\eta_1$ with $\upsilon(\eta_1)=  {i_1}\cdots i_{p-1} i_{p+2} \cdots {i_k}$. If $\eta_1$ and $\eta_2$ are curves with $\upsilon(\eta_1) = {i_1}\cdots {i_k}$ and $\upsilon(\eta_2) = {j_1} \cdots {j_\ell}$, define their {\em concatenation} $\eta_1\eta_2$ to be a curve such that  $\upsilon(\eta_1\eta_2)= {i_1}\cdots {i_k} {j_1} \cdots {j_\ell}$. 

\begin{Exa} \label{exa-pr}
Continuing Example~\ref{exa-1},
we choose admissible curves $\eta_i^\bv$ on a triangulated torus $\Sigma_\sigma$ such that $r_i^\bv = s(\eta_i^\bv)$ and draw the curves in Figure~\ref{fig-vague1} to illustrate that they are non-self-intersecting. This verifies Conjecture \ref{vague_conj-2} for this example. (In this example, it is not necessary to go through $\pi$.) We also draw the curves on the universal cover of $\Sigma_\sigma$ in Figure~\ref{fig-vague2} to see that they have no pairwise intersections. 
\end{Exa}
\begin{figure}
\begin{subfigure}{0.30\textwidth}
\centering
\includegraphics[scale=0.42]{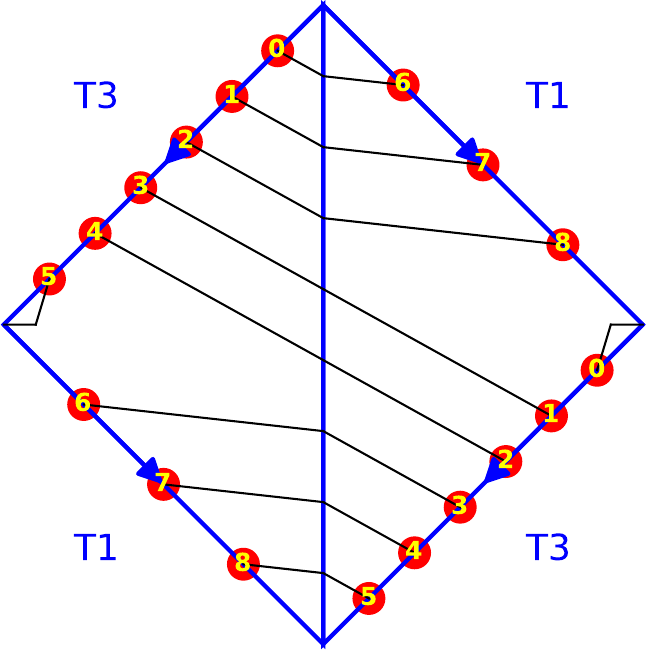}
\caption{The curve $\eta_1^\bv$.}
\end{subfigure}
\begin{subfigure}{0.30\textwidth}
\centering
\includegraphics[scale=0.42]{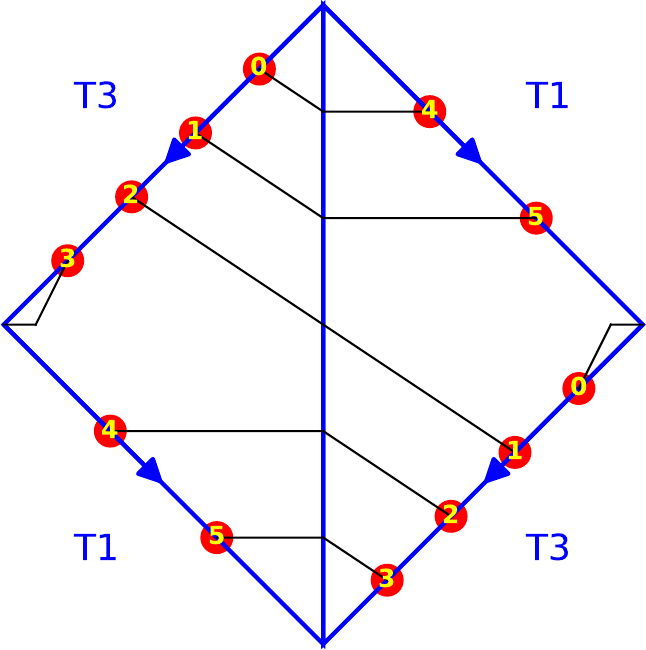}
\caption{The curve $\eta_2^\bv$.}
\end{subfigure}
\begin{subfigure}{0.30\textwidth}
\centering
\includegraphics[scale=0.42]{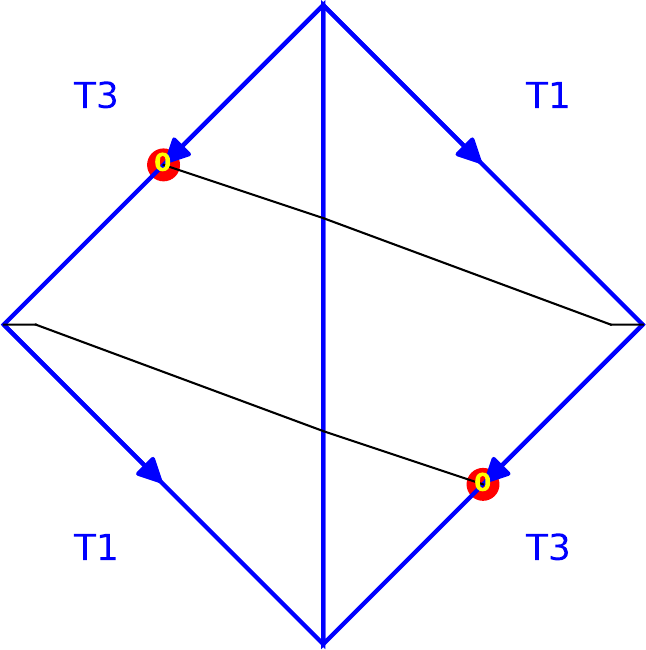}
\caption{The curve $\eta_3^\bv$.}
\end{subfigure}
\caption{The curves $\eta_i^\bv$ corresponding to Example~\ref{exa-1} displayed on $\Sigma_\sigma$ where $\sigma = (3,1,2) \in S_3$ written in one-line notation.} \label{fig-vague1}
\end{figure}
\begin{figure}
\begin{center}
\begin{tikzpicture}[scale=0.35mm]
\draw [help lines] (0,0) grid (7,4);
\draw [help lines] (0,1)--(1,0);
\draw [help lines] (0,2)--(2,0);
\draw [help lines] (0,3)--(3,0);
\draw [help lines] (0,4)--(4,0);
\draw [help lines] (1,4)--(5,0);
\draw [help lines] (2,4)--(6,0);
\draw [help lines] (3,4)--(7,0);
\draw [help lines] (4,4)--(7,1);
\draw [help lines] (5,4)--(7,2);
\draw [help lines] (6,4)--(7,3);
\draw [thick,red] (1,1)--(3,2);
\draw [thick,red] (1,1)--(0.79,0.92);  
\draw [thick,red] (1,1)--(0.93,0.90);
\draw [thick,red] (6,3)--(6+0.09,3+0.08);  
\draw [thick,red] (4,2)--(4+0.07,2+0.08); 
\draw [thick,red] (1.31,1.05)--(4-0.14,2-0.05);  
\draw [thick,red] (1.16,1.03)--(6-0.16,3-0.03);
\draw[thick,red,scale=0.3,domain=4.00:6.52,smooth,variable=\t]
plot ({(0.2+0.05*\t)*cos(\t r)+3.32},{(0.2+0.05*\t)*sin(\t r)+3.32});
\draw[thick,red,scale=0.6,domain=3.50:6.47,smooth,variable=\t]
plot ({(0.2+0.05*\t)*cos(\t r)+1.66},{(0.2+0.05*\t)*sin(\t r)+1.66});
\draw[thick,red,scale=0.3,domain=6.89:9.71,smooth,variable=\t]
plot ({(0.03+0.05*\t)*cos(\t r)+20},{(0.03+0.05*\t)*sin(\t r)+10});
\draw[thick,red,scale=0.3,domain=7.00:9.72,smooth,variable=\t]
plot ({(0.03+0.05*\t)*cos(\t r)+13.33},{(0.03+0.05*\t)*sin(\t r)+6.66});
\end{tikzpicture}
\end{center}
\caption{The curves from Example~\ref{exa-1}. The shortest curve corresponds to $\eta_3^\bv$, and the longest one to $\eta_1^\bv$.}
\label{fig-vague2}
\end{figure}
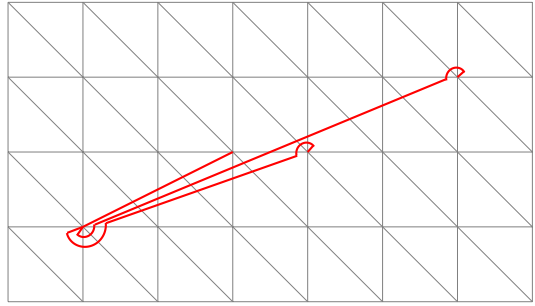

\begin{Def} \label{def-many}
For  $\mathfrak{v}=i_1i_2 \cdots i_k\in \mathfrak W$, define  $s(\mathfrak{v})=s_{i_1}...s_{i_k}\in \mathcal W \subset \mathcal A$.
We write $s(\eta)=s(\upsilon(\eta))$ for an admissible curve $\eta$.  
\end{Def}

Now we state Conjecture~\ref{vague_conj-2} in a more refined way.

\begin{Conj}[Conjecture~\ref{vague_conj-2}] \label{conj-main}
Fix an ordering on $\mathcal I$ so that a GIM $A$ is determined. Then, for each mutation sequence $\bw$, there exists a family of non-self-crossing admissible curves $\eta^{\bw }_i$, $i=1,\dots, n$, on the Riemann surface $\Sigma_\sigma$ for some $\sigma \in S_n$ such that 
$  \pi(r_i^\bw) = \pi \left (s(\eta_i^\bw) \right ).$
\end{Conj}

\begin{figure}
\begin{subfigure}{0.48\textwidth}
\centering
\includegraphics[scale=0.6]{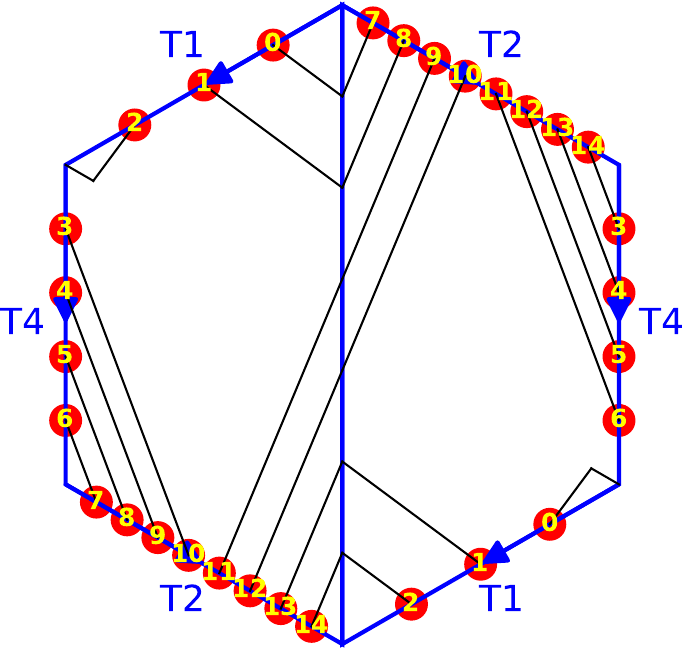}
\caption{The curve $\eta_1^\bv$.}
\end{subfigure}
\begin{subfigure}{0.48\textwidth}
\centering
\includegraphics[scale=0.6]{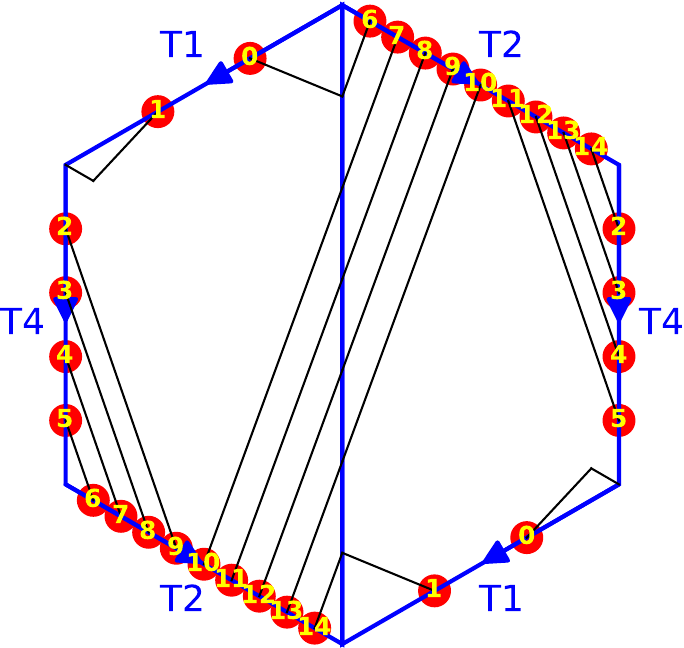}
\caption{The curve $\eta_2^\bv$.}
\end{subfigure}
\begin{subfigure}{0.48\textwidth}
\centering
\includegraphics[scale=0.6]{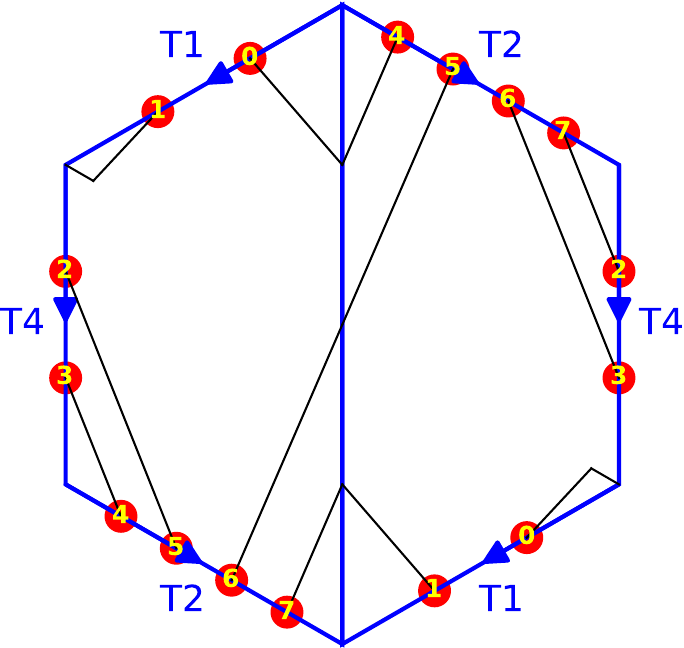}
\caption{The curve $\eta_3^\bv$.}
\end{subfigure}
\begin{subfigure}{0.48\textwidth}
\centering
\includegraphics[scale=0.6]{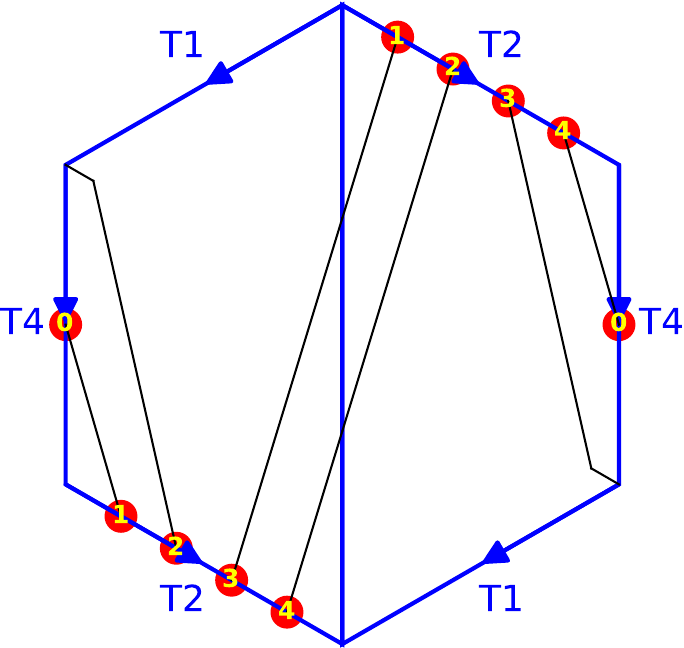}
\caption{The curve $\eta_4^\bv$.}
\end{subfigure}
\caption{The curves for Example~\ref{exa-all} drawn on $\Sigma_\sigma$ with $\sigma = (1,4,2,3)$.}\label{fig-rank4}
\end{figure}
\begin{Exa}\label{exa-all}
Consider the matrix $B = 
\begin{bmatrix}
0 & -1 & -1 & 2 \\
1 & 0 & 1 & -1 \\
1 & -1 & 0 & -1 \\
-2 & 1 & 1 & 0
\end{bmatrix}$.
 It arises from a triangulation of the torus with one boundary component with one marked point. It is commonly referred to as the {\em dreaded torus}. With the mutation sequence $\bw = [2, 3, 4, 2, 1, 3],$ we have 
$$\begin{bmatrix}
B & I \end{bmatrix} \qquad \xrightarrow{\ \bw \ } \qquad 
\begin{bmatrix}
0 & 1 & -1 & -1 & 0 & 2 & 3 & 2 \\
-1 & 0 & -1 & 2 & 2 & 3 & 3 & 2 \\
1 & 1 & 0 & -1 & -1 & -2 & -3 & -2 \\
1 & -2 & 1 & 0 & 0 & -2 & -2 & -1
\end{bmatrix}.$$

Choose the linear ordering $1\prec 3\prec 2\prec 4$. From Definition \ref{def-r}, we obtain
\begin{align*}
r_1^\bw & = s_{1}s_{3}(s_{2}s_{4}s_{2}s_{3})^2s_{1}(s_{3}s_{2}s_{4}s_{2})^2s_{3}s_{1} ,\\
r_2^\bw  &= s_{1}s_{3}(s_{2}s_{4}s_{2}s_{3})^2s_{2}(s_{3}s_{2}s_{4}s_{2})^2s_{3}s_{1}, \\
r_3^\bw &= s_{1}s_{3}s_{2}s_{4}s_{2}s_{3}s_{2}s_{4}s_{2}s_{3}s_{1}, \\
r_4^\bw &= s_{2}s_{3}s_{2}s_{4}s_{2}s_{3}s_{2}.
\end{align*}
In Figure~\ref{fig-rank4} we provide curves $\eta_i^\bw$ such that $s(\eta_i^\bw) = r_i^\bw$ for all $i \in \mathcal I.$ It is clear that they are non-self-intersecting on the surface $\Sigma_\sigma$ with $\sigma = (1,4,2,3) \in S_4$ written in one-line notation. By inspection these curves can be seen to be pairwise non-crossing.
\end{Exa}

In Example~\ref{exa-pi-nec} we show $\pi$ is necessary in Conjecture~\ref{conj-main} to avoid self-intersections.

\begin{Exa}\label{exa-pi-nec}
Consider the matrix $B = \begin{bmatrix}
0 & -2 & -2 & 3 \\
2 & 0 & 4 & 2 \\
2 & -4 & 0 & -1 \\
-3 & -2 & 1 & 0
\end{bmatrix}$. 
Applying to the mutation sequence $\bw=[4,3,1,4,1]$ we have 
$$ r_4^\bw= s_{3}s_{4}s_{1}(s_{4}s_{3})^2s_{4}s_{2}(s_{4}s_{3})^3s_{4}s_{2}s_{4}(s_{3}s_{4})^2s_{1}s_{4}s_{3} .$$
Let $\eta$ be the curve defined by $s(\eta) = r_4^\bw$. Upon inspection, for any
$\sigma \in S_4$ the curve $\eta$ has a self-intersection in $\Sigma_\sigma.$
However, for any choice of GIM we have $\pi( (s_3s_4)^3 ) = 1$ so the curve $\eta'$ given by $\upsilon(\eta')=34132423143 \in \mathfrak W$ satisfies $\pi( r_4^\bw) = \pi(s(\eta'))$ and can be drawn with no self-intersections.
\end{Exa}

In order to refine Conjecture \ref{vague_conj-3}, we need a new definition.
A sequence of indices $(i_1,\dots,i_d)$ is said to be a {\em chordless cycle} in a skew-symmetrizable matrix $B$ if \begin{enumerate}
\item $i_j = i_k$ if and only if $\{ j,k \} = \{ 1,d \}$,
\item for any distinct $j,k \in \{1,\dots,d\}$ we have $b_{i_j,i_k} \neq 0$ if and only if $|j- k|=1$,
\end{enumerate}
Additionally, a chordless cycle is said to be {\em oriented} if and only if all entries $b_{i_j,i_{j+1}}$ for $j=1,\dots, d-1$ have the same sign.
Two chordless cycles are considered equivalent if they have the same underlying set of indices. 

\begin{Conj}[Conjecture \ref{vague_conj-3}] \label{conj-3}
Let $B$ be a skew-symmetrizable matrix. 
\begin{enumerate}
\item There exists a linear ordering $\prec$ on $\mathcal I$ such that every oriented chordless cycle $(i_1,\dots,i_d)$ in $B$ has an odd number of positive $a_{i_j,i_{j+1}}$, $j=1,\dots,d-1$, where $A=[a_{ij}]$ is the GIM determined by $\prec$. 

\item Fix an ordering $\prec$ and its GIM $A$ satisfying the condition in (1). If $\bw$ and $\bv$ are two mutation sequences such that $C^\bw=C^\bv$ then $\pi(r_i^\bw) =\pi(r_i^\bv), \,  i=1,\dots,n. $
\end{enumerate}
\end{Conj}

The elements $\pi(r_i^\bw)$ can be viewed as elements of $\pi(\mathcal W)$, and Conjecture~\ref{conj-3} can be interpreted as a statement about relations in $\pi(\mathcal W).$ Relations for these groups have been explored for particular skew-symmetrizable matrices and  a restricted class of GIMs in \cite{BMa,FT2,Se3}.
A thorough investigation of relations in $\pi(\mathcal W)$ and their application to Conjecture~\ref{conj-3} will take place in a subsequent article. It is expected that all of the discovered relations will hold for any GIM satisfying the condition in Conjecture~\ref{conj-3} (1) which is a weaker than Seven's notion of admissibility \cite{Se,Se3}. 

\medskip

In Proposition \ref{prop:mutation-finite-ordering} below, we will prove Conjecture~\ref{conj-3} (1) for a special family using results in \cite{Se, Se1}. In discussing the notion of cycles we will briefly switch from the perspective of matrices to that of the directed graph.

\begin{Def}
Let $B$ be an $n \times n$ skew-symmetrizable matrix. Define $\mathcal{G}(B)$ to be the directed graph with vertices in $\mathcal I$ and arrows $i \rightarrow j$ for $b_{ij} <0$.
\end{Def}

Note that the definition of a chordless cycle for a matirx $B$ is equivalent to the standard definition of chordless cycle in the directed graph $\mathcal{G}(B)$.

\medskip

Now, for the time being, assume that $B=[b_{ij}]$ is a skew-symmetrizable matrix which can be mutated from an acyclic matrix $B_0$ through a mutation sequence $\bw$, i.e., assume $B=B_0^\bw$. Let $A_0$ be the generalized Cartan matrix associated with $B_0$, and define \begin{equation} \label{cbw}  A=[a_{ij}]:=C^\bw A_0 (C^\bw)^\top.\end{equation} Then, by \cite[Theorems 1.2]{Se1} (see also \cite{Se}), the matrix $A$ is a GIM such that $|a_{ij}|=|b_{ij}|$ for $i \neq j$ and
\begin{equation} \label{pro-n}
\text{every oriented chordless cycle of $\mathcal G(B)$ has exactly one edge $\{i,j\}$ such that $a_{ij} > 0.$}
\end{equation}

Let us consider the following conditions for $\mathcal G(B)$:
\begin{itemize}
\item[(AC1)] every oriented (not necessarily chordless) cycle has at least one edge $\{i,j\}$ such that $a_{ij} > 0$; 
\item[(AC2)] if an edge $\{i,j\}$ with $a_{ij}>0$ is contained in a cycle either oreinted or non-oriented, then it is also contained in an oriented chordless cycle.
\end{itemize}

\begin{Prop}\label{prop:mutation-finite-ordering}
Assume that $B$ is a skew-symmetrizable matrix which can be mutated from an acyclic matrix $B_0$. Let $A=[a_{ij}]$ be the GIM defined in \eqref{cbw}. Suppose that (AC1) and (AC2) hold. Then Conjecture~\ref{conj-3} (1) is true. 
\end{Prop}

\begin{proof}
It follows from \eqref{pro-n} that $A$ satisfies Conjecture~\ref{conj-3} (1) if it arises from a linear ordering. To this effect, let $\mathcal G= \mathcal G(B)$, and define $\mathcal G^\circ$ to be the graph obtained from $\mathcal G$ by reversing the directions of edges $\{i,j\}$ with $a_{ij} >0$. We will show that $\mathcal G^\circ$ is acyclic,  and define a relation $\prec$ on the set $\mathcal{I}$ of vertices as follows: \begin{equation*} \text{$i \prec j$ if there is a directed path $i=i_1 \rightarrow \dots \rightarrow i_p =j$ in $\mathcal G^\circ$.}
\end{equation*}
Then the relation $\prec$ will be a strict partial order on $\mathcal{I}$. 

Suppose that there is an oriented cycle $E_0=(i_0\rightarrow i_1\rightarrow \dots \rightarrow i_p=i_0)$ in $\mathcal G^\circ$. Then it is also a cycle in $\mathcal G$, but not necessarily oriented. We inductively define the sequence $E_0,E_1,...,E_p$ of oriented cycles in $\mathcal G^\circ$ as follows: Suppose that $E_d$ is defined for some $d\in\{0,1,...,p-1\}$.  If $a_{i_d, i_{d+1}} < 0$ then we define $E_{d+1}$ to be equal to $E_d$. Suppose that $a_{i_d, i_{d+1}} > 0$.
By (AC2), there must be an oriented chordless cycle $(i_d\rightarrow j_1\rightarrow j_2\rightarrow \dots\rightarrow j_r\rightarrow i_{d+1}\rightarrow i_d)$ in $\mathcal G$. Then we define $E_{d+1}$ as a subgraph of $\mathcal G^\circ$  to be the oriented cycle obtained from $E_d$ by replacing the single arrow $i_d \rightarrow  i_{d+1}$ with the oriented path $i_d \rightarrow j_1 \dots \rightarrow j_r \rightarrow i_{d+1}$. Here, thanks to  \eqref{pro-n},  we have $a_{i_d,j_1}<0$, $a_{j_e,j_{e+1}}< 0$ for $e\in\{1,...,r-1\}$, and $a_{j_r,i_{d+1}}<0$. Once $E_0,E_1,...,E_p$ are defined, the last one $E_p$ is an oriented cycle $(k_0\rightarrow k_1 \rightarrow \dots\rightarrow k_s=k_0)$ such that $\{i_0,...,i_{p-1}\}\subseteq\{k_0,...,k_{s-1}\}$ and $a_{k_e,k_{e+1}}< 0$ for all $e=0, \dots , s-1$. By definition of $\mathcal G^\circ$, the graph $\mathcal G$ also has the same oriented cycle $(k_0\rightarrow k_1 \rightarrow \dots\rightarrow k_s=k_0)$. This contradicts (AC1). Thus $\mathcal G^\circ$ is acyclic.   

Now refine $\prec$ to a linear ordering on $\mathcal I$. Let $\widetilde A=[ \tilde a_{ij}]$ be given by \eqref{eqn-gim}. We need to show that $\widetilde A=A$. We have $\tilde a_{ij} =a_{ij}=2$ if $i =j$, and  $\tilde a_{ij} =a_{ij}=0$ if $b_{ij}=0$. Assume $i \prec j$ and $\tilde a_{ij}=b_{ij} <0$. If $a_{ij}>0$, then $j \prec i$ by definition, which is a contradiction. Thus $a_{ij}<0$ and $\tilde a_{ij} =a_{ij}$. Assume $i \prec j$ and $\tilde a_{ij}=b_{ij} >0$. Then $b_{ji}<0$. If $a_{ij}<0$, then $a_{ji}<0$ and hence $j \prec i$ by definition, which is a contradiction. Thus $a_{ij}>0$ and $\tilde a_{ij} =a_{ij}$. The other cases are similar, and we have $\tilde a_{ij}=a_{ij}$ in all the cases.
\end{proof}

\begin{Exa} \label{ex-via}
Let $B=[b_{ij}]$ be the skew-symmetric matrix associated with the quiver $Q$ below via the rule $b_{ij}=-1$ if $i \rightarrow j$ and $b_{ij}=0$ if there is no arrow between $i$ and $j$.  
This quiver is obtained applying mutations at vertices $6,5,3,4$ to the acyclic quiver $Q_0$ also shown below.
\[ 
Q=\raisebox{-1.5 cm}{\begin{tikzpicture}[scale=0.4]
\centering
\tkzDefPoint(0,0){A}
\tkzDefShiftPoint[A](60:2.5){F}
\tkzDefShiftPoint[A](0:2.42){B}
\tkzDefShiftPoint[A](-120:2.5){C}
\tkzDefShiftPoint[A](-60:2.5){D}
\tkzDefShiftPoint[B](-60:2.5){E}
\tkzDrawPoints(A,B,C,D,E,F)
\tkzLabelPoint[above right](B){\tiny $3$}
\tkzLabelPoint[above left](A){\tiny $2$}
\tkzLabelPoint[below left](C){\tiny $4$}
\tkzLabelPoint[below](D){\tiny $5$}
\tkzLabelPoint[above](F){\tiny $1$}
\tkzLabelPoint[below right](E){\tiny $6$}
\tkzDrawSegments[->, shorten >=5, shorten <=5, >=stealth'](A,F F,B B,A C,A D,C A,D E,D D,B B,E)
\end{tikzpicture}}
\qquad \qquad Q_0=\raisebox{-1 cm}{\begin{tikzpicture}[scale=0.4]
\centering
\tkzDefPoint(0,0){A}
\tkzDefShiftPoint[A](0:5){F}
\tkzDefShiftPoint[A](0:2.42){B}
\tkzDefShiftPoint[A](-120:2.5){C}
\tkzDefShiftPoint[A](-60:2.5){D}
\tkzDefShiftPoint[B](-60:2.5){E}
\tkzDrawPoints(A,B,C,D,E,F)
\tkzLabelPoint[above right](B){\tiny $3$}
\tkzLabelPoint[above left](A){\tiny $2$}
\tkzLabelPoint[below left](C){\tiny $4$}
\tkzLabelPoint[below](D){\tiny $5$}
\tkzLabelPoint[above](F){\tiny $1$}
\tkzLabelPoint[below right](E){\tiny $6$}
\tkzDrawSegments[->, shorten >=5, shorten <=5, >=stealth'](A,D D,C D,B B,E E,F)
\end{tikzpicture}}
\]
From \eqref{cbw}, we obtain GIM $A=[a_{ij}]= {\tiny \begin{bmatrix}
2 & -1 & 1 & 0 & 0 &0\\
-1 & 2 &-1 & -1 & 1 &0\\
1 &-1 & 2 & 0 &-1&-1 \\
0 &-1 & 0 & 2 & -1&0 \\
0 & 1 &-1 &-1 & 2&1\\0&0&-1&0&1&2
\end{bmatrix}}$
associated to $B$ (or $Q$). We specify the signature of $a_{ij}$ on $Q(=\mathcal G)$ and draw the acyclic graph $\mathcal G^\circ$ defined in the proof of Proposition \ref{prop:mutation-finite-ordering}:
\[ 
\raisebox{-1.2 cm}{\begin{tikzpicture}[scale=0.4]
\centering
\tkzDefPoint(0,0){A}
\tkzDefShiftPoint[A](60:2.5){F}
\tkzDefShiftPoint[A](0:2.42){B}
\tkzDefShiftPoint[A](-120:2.5){C}
\tkzDefShiftPoint[A](-60:2.5){D}
\tkzDefShiftPoint[B](-60:2.5){E}
\tkzDrawPoints(A,B,C,D,E,F)
\tkzLabelPoint[above right](B){\tiny $3$}
\tkzLabelPoint[above left](A){\tiny $2$}
\tkzLabelPoint[below left](C){\tiny $4$}
\tkzLabelPoint[below](D){\tiny $5$}
\tkzLabelPoint[above](F){\tiny $1$}
\tkzLabelPoint[below right](E){\tiny $6$}
\tkzDrawSegments[->, shorten >=5, shorten <=5, >=stealth'](A,F F,B B,A C,A D,C A,D E,D D,B B,E)
\path (0.2,1) node[red, node contents=\tiny{$-$}]
(-1.2,-1.2) node[red, node contents=\tiny{$-$}]
(1.4,0.2) node[red, node contents=\tiny{$-$}]
(0.2,-2.4) node[red, node contents=\tiny{$-$}]
(3.5,-1.2) node[red, node contents=\tiny{$-$}]
(2.1,-1.2) node[red, node contents=\tiny{$-$}]
(2.2,1.2) node[blue, node contents=\tiny{$+$}]
(0.4,-1.2) node[blue, node contents=\tiny{$+$}]
(2.5,-2.4) node[blue, node contents=\tiny{$+$}]; 
\end{tikzpicture}} 
\qquad \qquad \mathcal G^\circ=\raisebox{-1.2 cm}{\begin{tikzpicture}[scale=0.4]
\centering
\tkzDefPoint(0,0){A}
\tkzDefShiftPoint[A](60:2.5){F}
\tkzDefShiftPoint[A](0:2.42){B}
\tkzDefShiftPoint[A](-120:2.5){C}
\tkzDefShiftPoint[A](-60:2.5){D}
\tkzDefShiftPoint[B](-60:2.5){E}
\tkzDrawPoints(A,B,C,D,E,F)
\tkzLabelPoint[above right](B){\tiny $3$}
\tkzLabelPoint[above left](A){\tiny $2$}
\tkzLabelPoint[below left](C){\tiny $4$}
\tkzLabelPoint[below](D){\tiny $5$}
\tkzLabelPoint[above](F){\tiny $1$}
\tkzLabelPoint[below right](E){\tiny $6$}
\tkzDrawSegments[->, shorten >=5, shorten <=5, >=stealth'](A,F B,F B,A C,A D,C D,A D,E D,B B,E)
\end{tikzpicture}} 
\]
It is easy to see that $\mathcal G$ satisfies (AC1) and (AC2). Indeed, we see \eqref{pro-n} holds, and there is only one additional (simple) oriented cycle $(1,3,6,5,4,2,1)$ with chords, which has two positive edges. Now the definition of $\prec$ in the proof of Proposition \ref{prop:mutation-finite-ordering} yields $5 \prec 4 \prec 2 \prec 1 $, $5 \prec 3 \prec 2 \prec 1$ and $5 \prec 3 \prec 6$. Thus a refinement to a linear odering is given by $5 \prec 4 \prec 3 \prec 6 \prec 2 \prec 1 $, which gives rise to $A$ via \eqref{eqn-gim}.   Clearly,  Conjecture~\ref{conj-3} (1) holds with this linear ordering. \end{Exa}

\begin{Exa}
Let $B$ be the skew-symmetric matrix associated with the quiver $Q$ below in the same way as in Example \ref{ex-via}.  
This quiver is obtained applying mutations at vertices $5,3,4$ to the acyclic quiver $Q_0$ also shown below.
\[ 
Q=\raisebox{-1.2 cm}{\begin{tikzpicture}[scale=0.3]
\centering
\tkzDefPoint(0,0){A}
\tkzDefShiftPoint[A](12,0){F}
\tkzDefShiftPoint[A](4,1.5){B}
\tkzDefShiftPoint[A](8,1.5){C}
\tkzDefShiftPoint[A](6,4){D}
\tkzDefShiftPoint[A](6,6){E}
\tkzDrawPoints(A,B,C,D,E,F)
\tkzLabelPoint[below](B){\tiny $3$}
\tkzLabelPoint[below](A){\tiny $1$}
\tkzLabelPoint[below](C){\tiny $5$}
\tkzLabelPoint[below](D){\tiny $4$}
\tkzLabelPoint[below](F){\tiny $6$}
\tkzLabelPoint[above](E){\tiny $2$}
\tkzDrawSegments[->, shorten >=5, shorten <=5, >=stealth'](A,E A,F D,A D,C C,B B,D F,D E,F D,E)
\end{tikzpicture}}
\qquad \qquad Q_0=\raisebox{-1.2 cm}{\begin{tikzpicture}[scale=0.4]
\centering
\tkzDefPoint(0,0){A}
\tkzDefShiftPoint[A](12,0){F}
\tkzDefShiftPoint[A](4,0){B}
\tkzDefShiftPoint[A](8,0){C}
\tkzDefShiftPoint[A](6,4){D}
\tkzDefShiftPoint[A](2,4){E}
\tkzDrawPoints(A,B,C,D,E,F)
\tkzLabelPoint[below](B){\tiny $3$}
\tkzLabelPoint[below](A){\tiny $1$}
\tkzLabelPoint[below](C){\tiny $5$}
\tkzLabelPoint[above](D){\tiny $4$}
\tkzLabelPoint[below](F){\tiny $6$}
\tkzLabelPoint[above](E){\tiny $2$}
\tkzDrawSegments[->, shorten >=5, shorten <=5, >=stealth'](A,E A,B E,B B,D D,C C,F)
\end{tikzpicture}}
\]
From \eqref{cbw}, we obtain GIM $A=[a_{ij}]= {\tiny \begin{bmatrix}
2 & -1 & 0 & 1 & 0 & -1\\
-1 & 2 &0 & 1 & 0 &-1\\
0 &0 & 2 & -1 &1&0 \\
1 &1 & -1 & 2 & -1&-1 \\
0 & 0 &1 &-1 & 2&0\\
-1&-1&0&-1&0&2
\end{bmatrix}}$. We specify the signature of $a_{ij}$ on $Q (=\mathcal G)$ and draw the acyclic graph $\mathcal G^\circ$:
\[ 
\raisebox{-1 cm}{\begin{tikzpicture}[scale=0.3]
\centering
\tkzDefPoint(0,0){A}
\tkzDefShiftPoint[A](12,0){F}
\tkzDefShiftPoint[A](4,1.5){B}
\tkzDefShiftPoint[A](8,1.5){C}
\tkzDefShiftPoint[A](6,4){D}
\tkzDefShiftPoint[A](6,6){E}
\tkzDrawPoints(A,B,C,D,E,F)
\tkzLabelPoint[below](B){\tiny $3$}
\tkzLabelPoint[above left](A){\tiny $1$}
\tkzLabelPoint[below](C){\tiny $5$}
\tkzLabelPoint[below](D){\tiny $4$}
\tkzLabelPoint[above right](F){\tiny $6$}
\tkzLabelPoint[above](E){\tiny $2$}
\tkzDrawSegments[->, shorten >=5, shorten <=5, >=stealth'](A,E A,F D,A D,C C,B B,D F,D E,F D,E)
\path (2.5,3) node[red, node contents=\scriptsize{$-$}]
(9.5,3) node[red, node contents=\scriptsize{$-$}]
(6,-0.3) node[red, node contents=\scriptsize{$-$}]
(6,1.2) node[blue, node contents=\scriptsize{$+$}]
(6.3,4.7) node[blue, node contents=\scriptsize{$+$}]
(4,3) node[blue, node contents=\scriptsize{$+$}]
(8,3) node[red, node contents=\scriptsize{$-$}]
(7,2.3) node[red, node contents=\scriptsize{$-$}]
(5,2.3) node[red, node contents=\scriptsize{$-$}]
; 
\end{tikzpicture}} 
\qquad \qquad \mathcal G^\circ= \raisebox{-1 cm}{\begin{tikzpicture}[scale=0.3]
\centering
\tkzDefPoint(0,0){A}
\tkzDefShiftPoint[A](12,0){F}
\tkzDefShiftPoint[A](4,1.5){B}
\tkzDefShiftPoint[A](8,1.5){C}
\tkzDefShiftPoint[A](6,4){D}
\tkzDefShiftPoint[A](6,6){E}
\tkzDrawPoints(A,B,C,D,E,F)
\tkzLabelPoint[below](B){\tiny $3$}
\tkzLabelPoint[above left](A){\tiny $1$}
\tkzLabelPoint[below](C){\tiny $5$}
\tkzLabelPoint[below](D){\tiny $4$}
\tkzLabelPoint[above right](F){\tiny $6$}
\tkzLabelPoint[above](E){\tiny $2$}
\tkzDrawSegments[->, shorten >=5, shorten <=5, >=stealth'](A,E A,F A,D D,C B,C B,D F,D E,F E,D)
\end{tikzpicture}} 
\]
It is straightforward to check that $\mathcal G$ satisfies (AC1) and (AC2), and we can take $1 \prec 2 \prec 3 \prec 6 \prec 4 \prec 5$ for Conjecture~\ref{conj-3} (1). \end{Exa}

\begin{Rmk}
It will be interesting to investigate when a skew-symmetrizable matrix mutated from an acyclic matrix satisfies (AC1) and (AC2). It may be that such a matrix always satisfies the conditions. 
\end{Rmk}

The lemma below provides another sufficient condition for existence of a linear ordering $\prec$ and its GIM $A$ satisfying the condition in Conjecture~\ref{conj-3} (1). If we do not require that a GIM is determined by a linear ordering, 
it can be proven that a GIM satisfying the condition of Conjecture~\ref{conj-3} (1) always exists for any skew-symmetrizable matrix.
But in order to define the elements $s_i^\bw \in \mathcal A$ as in the next section, it is necessary that $A$ arises from a linear ordering.

\begin{Lem}\label{lem-span-tree}
Let $B$ be a skew-symmetrizable matrix. Consider $\mathcal G=\mathcal{G}(B)$ as undirected. Assume that each of the (undirected) chordless cycles in $\mathcal G$ has an edge in the cycle that is not contained in any other (undirected) chordless cycles. Then Conjecture~\ref{conj-3} (1) is true.
\end{Lem}

\begin{proof}
For a collection of arrows $\mathcal E= \{e_1,\dots,e_p\}$ in $\mathcal G$, we can define a new directed graph $\mathcal H$ by reversing the direction of the arrows of $\mathcal E$. If $\mathcal H$ is acyclic we may define a linear order by setting $i \prec j$ if $i \rightarrow j$ is an arrow of $\mathcal H$ and extending it to a linear ordering on $\mathcal{I}$. We will show that there exists a set of arrows that contains an odd number of arrows (actually one arrow) from every oriented chordless cycle of $\mathcal G$ such that $\mathcal H$ is acyclic. Therefore it follows from \eqref{eqn-gim} that the associated GIM satisfies the condition in the statement of the lemma.

As in the statement of the lemma, we consider $\mathcal G$ {\em undirected} for the time being. Let $\{\mathcal C_1, \mathcal C_2, \dots , \mathcal C_s\}$ be the set of undirected chordless cycles in $\mathcal G$ and take $\mathcal E'=\{ e_1, e_2, \dots , e_s \}$ to be the set of edges in $\mathcal G$ such that $e_i$ is an edge of $\mathcal C_i$ and not an edge of $\mathcal C_j$ for any $j \neq i$. Such an $\mathcal E'$ exists by the assumption. Let $\mathcal T$ be the spanning tree obtained from removing the edges in $\mathcal E'$ from $\mathcal G$. Now we consider $\mathcal G$ directed again, and let $\bar e_i$ be the opposite arrow of $e_i.$ We will construct the desired sequence $\mathcal E$ of arrows as a subset of $\mathcal E'$ by iteratively taking $e_i$ to be in $\mathcal E$ if and only if either 
\begin{enumerate}
\item $\mathcal C_i$ is oriented in $\mathcal G$, or  
\item $\mathcal T \cup \{\bar e_k | e_k \in \mathcal E, k < i\} \cup \{e_i\}$ has an oriented cycle. 
\end{enumerate}

Now define $\mathcal H$ from $\mathcal G$ by reversing the direction of the arrows of $\mathcal E$. Then for any oriented cycle of $\mathcal G$ we have reversed only one arrow of the cycle by (1) and the choice of $\mathcal E'$, so any oriented chordless cycle of $\mathcal G$ is no longer oriented in $\mathcal H$.  Furthermore every non-oriented cycle of $\mathcal G$ remains non-oriented in $\mathcal H$ by (2). Therefore all of the chordless cycles of $\mathcal H$ are non-oriented and it must be that $\mathcal H$ is acyclic. 
\end{proof}

We now give an example illustrating the proof of Lemma~\ref{lem-span-tree}.

\begin{Exa}\label{exa-span-tree}
Let $B$ be the skew-symmetric matrix given in Figure~\ref{fig-span-tree}, or any skew-symmetric matrix with the same directed graph $\mathcal G$ shown in the figure. The graph $\mathcal G$ has two oriented chordless cycles $(1,3,4,1)$ and $(2,4,5,2)$, and three undirected chordless cycles $\mathcal C_1, \mathcal C_2$ and $\mathcal C_3$ given by $\{1,3,4\}$, $\{1,2,4\}$, and $\{2,4,5\}$, respectively. Consider $e_1 = 3 \rightarrow 1$, $e_2 = 1\rightarrow 2$, and $e_3 = 5 \rightarrow 2$. Then $\mathcal E'=\{ e_1,e_2, e_3\}$ satisfies the assumption of Lemma~\ref{lem-span-tree}, and we obtain the spanning tree $\mathcal T = \raisebox{-1 cm}{\begin{tikzpicture}[scale=0.5]
\centering \tkzDefPoint(0,0){A} \tkzDefShiftPoint[A](0:2.42){B} \tkzDefShiftPoint[A](-120:2.5){C} \tkzDefShiftPoint[A](-60:2.5){D} \tkzDefShiftPoint[B](-60:2.5){E} \tkzDrawPoints(A,B,C,D,E) \tkzLabelPoint[above right](B){ \tiny $2$} \tkzLabelPoint[above left](A){ \tiny $1$} \tkzLabelPoint[below left](C){ \tiny $3$} \tkzLabelPoint[below](D){ \tiny $4$} \tkzLabelPoint[below right](E){ \tiny $5$} \tkzDrawSegments[->, shorten >=5, shorten <=5, >=stealth'](D,C A,D D,E B,D) \end{tikzpicture}} $ by removing $\mathcal E'$ from $\mathcal G$. Now to construct $\mathcal E$ we see that $e_1 \in \mathcal E$ by condition (1), $e_2 \not\in \mathcal E$ since $\mathcal C_2$ is not oriented and $\mathcal T \cup \{ \overline{e_1}, e_2 \} =  \raisebox{-1 cm}{\begin{tikzpicture}[scale=0.5]
\centering \tkzDefPoint(0,0){A} \tkzDefShiftPoint[A](0:2.42){B} \tkzDefShiftPoint[A](-120:2.5){C} \tkzDefShiftPoint[A](-60:2.5){D} \tkzDefShiftPoint[B](-60:2.5){E} \tkzDrawPoints(A,B,C,D,E) \tkzLabelPoint[above right](B){ \tiny $2$} \tkzLabelPoint[above left](A){ \tiny $1$} \tkzLabelPoint[below left](C){ \tiny $3$} \tkzLabelPoint[below](D){ \tiny $4$} \tkzLabelPoint[below right](E){ \tiny $5$} \tkzDrawSegments[->, shorten >=5, shorten <=5, >=stealth'](A,C A,B D,C A,D D,E B,D) \end{tikzpicture}} $ does not have an oriented cycle, and $e_3 \in \mathcal E$ by condition (1). Thus $\mathcal E=\{ e_1, e_2 \}$, and $\mathcal H= \raisebox{-1 cm}{\begin{tikzpicture}[scale=0.5]
\centering \tkzDefPoint(0,0){A} \tkzDefShiftPoint[A](0:2.42){B} \tkzDefShiftPoint[A](-120:2.5){C} \tkzDefShiftPoint[A](-60:2.5){D} \tkzDefShiftPoint[B](-60:2.5){E} \tkzDrawPoints(A,B,C,D,E) \tkzLabelPoint[above right](B){ \tiny $2$} \tkzLabelPoint[above left](A){ \tiny $1$} \tkzLabelPoint[below left](C){ \tiny $3$} \tkzLabelPoint[below](D){ \tiny $4$} \tkzLabelPoint[below right](E){ \tiny $5$} \tkzDrawSegments[->, shorten >=5, shorten <=5, >=stealth'](A,C A,B D,C A,D D,E B,D B,E) \end{tikzpicture}} $. The covering relations dictated by the acyclic graph $\mathcal H$ are $ 1 \prec 4 \prec 3$, $1 \prec 2 \prec 4$, and $2 \prec 4 \prec 5.$ One extension of these relations to a linear ordering is $1 \prec 2 \prec 4 \prec 3 \prec 5.$ It is straightforward to check that the associated GIM has exactly one positive entry for each oriented chordless cycle of $B$ (or of $\mathcal G$). 
\begin{figure}
\centering
\begin{minipage}{0.40\textwidth}
\centering
${\scriptsize \begin{bmatrix}
0 & -1 & 2 & -3 & 0 \\
1 & 0 & 0 & -4 & 5 \\
-2 & 0 & 0 & 6 & 0 \\
3 & 4 & -6 & 0 & -7 \\
0 & -5 & 0 & 7 & 0
\end{bmatrix}}$
\end{minipage}
\begin{minipage}{0.40\textwidth}
\hspace{0 cm}
\begin{tikzpicture}[scale=0.8]
\centering
\tkzDefPoint(0,0){A}
\tkzDefShiftPoint[A](0:2.42){B}
\tkzDefShiftPoint[A](-120:2.5){C}
\tkzDefShiftPoint[A](-60:2.5){D}
\tkzDefShiftPoint[B](-60:2.5){E}
\tkzDrawPoints(A,B,C,D,E)
\tkzLabelPoint[above right](B){\scriptsize $2$}
\tkzLabelPoint[above left](A){\scriptsize $1$}
\tkzLabelPoint[below left](C){\scriptsize $3$}
\tkzLabelPoint[below](D){\scriptsize $4$}
\tkzLabelPoint[below right](E){\scriptsize $5$}
\tkzDrawSegments[->, shorten >=5, shorten <=5, >=stealth'](A,B C,A D,C A,D D,E B,D E,B)
\end{tikzpicture}
\end{minipage}
\caption{A skew-symmetric matrix $B$ and the digraph associated to it in Lemma~\ref{lem-span-tree}. The proof of the lemma is illustrated in Example~\ref{exa-span-tree}.}\label{fig-span-tree}
\end{figure}
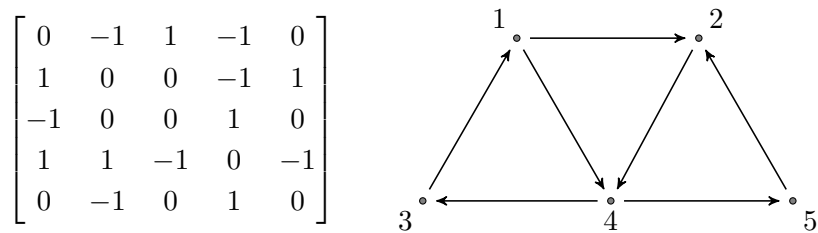

\end{Exa}

Recall the definition of an $L$-matrix from Definition \ref{def-ell}.
We now provide an example illustrating Conjecture~\ref{conj-3} and $l$-vectors.

\begin{Exa}\label{exa-pi-wv}
Let $B$ be the matrix from Example~\ref{exa-all}. For the two mutation sequences $\bw =[3, 4, 1, 3, 4, 3]$ and $\bv = [4, 1, 3, 4, 1, 3]$ we have $C^\bw = C^\bv.$ 
On the other hand, 
\begin{align*}
 {r_1^\bw} =& s_{3}s_{4}s_{3}s_{1}s_{3}s_{4}s_{3},\\
 {r_2^\bw} =& s_{3}s_{4}s_{3}s_{1}s_{3}s_{4}s_{2}s_{4}s_{3}s_{1}s_{3}s_{4}s_{3},\\
 {r_3^\bw} =& s_{3}s_{4}s_{1}s_{3}s_{4}s_{3}s_{1}s_{3}s_{1}s_{3}s_{4}s_{3}s_{1}s_{4}s_{3},\\
 {r_4^\bw} =& s_{3}s_{4}s_{1}s_{3}s_{4}(s_{3}s_{1})^2s_{3}s_{4}s_{3}(s_{1}s_{3})^2s_{4}s_{3}s_{1}s_{4}s_{3},
\end{align*}
 and 
 \begin{align*}
  {r_1^\bv} =& s_{3}(s_{4}s_{1})^2s_{4}s_{3}s_{4}s_{1}s_{4}s_{3}s_{4}(s_{1}s_{4})^2s_{3},\\
 {r_2^\bv} =& s_{3}(s_{4}s_{1})^2s_{4}s_{3}s_{4}s_{1}s_{4}s_{3}(s_{4}s_{1})^2s_{4}s_{2}s_{4}(s_{1}s_{4})^2s_{3}s_{4}s_{1}s_{4}s_{3}s_{4}(s_{1}s_{4})^2s_{3},\\
 {r_3^\bv} =& s_{3}(s_{4}s_{1})^2s_{4}s_{3}s_{4}(s_{1}s_{4})^2s_{3},\\
 {r_4^\bv} =&(s_{3}s_{4}s_{1})^2s_{4}(s_{1}s_{4}s_{3})^2.
 \end{align*}
There are two oriented cycles on vertices $\{ 1,4,2 \}$ and $\{ 1,4,3 \}$ in $B$. Take the GIM arising from the linear ordering $1 \prec 2 \prec 3 \prec 4$. Then only the entry $a_{14}$ is positive for the cycles, and the condition in Corollary \ref{vague_conj-3} is satisfied. 
Direct computation shows that $\pi ( {r_i^\bw}) = \pi ({r_i^\bv})$, and  Conjecture \ref{vague_conj-3} is verified. 

We identify $\alpha_i$ with $\lambda_i$ in Definition \ref{def-ell} and compute the $l$-vectors 
\begin{align*}
{l_1^\bw} =& s_{3}s_{4}s_{3}(\lambda_{1})=(1,0,-1,-1),&
{l_2^\bw} =& s_{3}s_{4}s_{3}s_{1}s_{3}s_{4}(\lambda_{2})=(-1,1,0,1),\\
{l_3^\bw} =& s_{3}s_{4}s_{1}s_{3}s_{4}s_{3}s_{1}(\lambda_{3})= (2,0,0,-3),&
{l_4^\bw} =& s_{3}s_{4}s_{1}s_{3}s_{4}(s_{3}s_{1})^2s_{3}(\lambda_{4})=(-3,0,0,4),
\end{align*}
and obtain the $L$-matrix
\[ L^\bw = \begin{bmatrix} 1&0&-1&-1\\ -1&1&0&1 \\2&0&0&-3\\ -3&0&0&4 \end{bmatrix}.\] 

On the other hand,
\begin{align*}
{l_1^\bv} =& (-1,0,1,1)=-l_1^\bw,&
{l_2^\bv} =& (-1,1,0,1)=l_2^\bw,\\
{l_3^\bv} =&  (-2,0,0,3)=-l_3^\bw,&
{l_4^\bv} =& (-3,0,0,4)=l_4^\bw.
\end{align*}

\end{Exa}

One may hope that the reflections $r_i^\bw$ would give a direct generalization of \cite[Theorem 1.4]{ST} with the expectation that a product of $r_i^{\bw}$'s 
might equal $s_{\tilde \sigma(1)}s_{\tilde \sigma(2)} \cdots s_{\tilde \sigma(n)}$ in $\mathcal W$ for some $\tilde \sigma \in S_n$. However Example~\ref{ex-counterexa} provides a counterexample.

\begin{Exa}\label{ex-counterexa}
Let $B$ be the matrix from Example~\ref{exa-pi-nec}. After the mutation sequence $\bw=  [2, 3, 2, 1]$ we have  
$$
 {r_1^\bw}=s_{1},\quad
 {r_2^\bw}=s_{1}s_{2}s_{1},\quad
 {r_3^\bw}=s_{2}s_{3}s_{2},\quad
 {r_4^\bw}=s_{3}s_{4}s_{3}.$$
It is straightforward to check that $\displaystyle \prod_{i \in \mathcal I}  {r_{\sigma(i)}^\bw}\neq {s_{\tilde\sigma (1)}} {s_{\tilde\sigma (2)}} {s_{\tilde\sigma (3)}}{s_{\tilde\sigma(4)}}$ for any pair of $\sigma, \tilde \sigma \in S_4.$ The same is true when considering the matrix representation of the $s_i$ for any choice of GIM associated to $B$. 

This collection $\{r_i^\bw\}$ also provides an example where for any $\sigma \in S_4$ there will always be some pair of curves in $\eta_i^\bw$ and $\eta_j^\bw$ satisfying Conjecture~\ref{conj-main} that intersect.
\end{Exa}

\section{Main Theorem} \label{sec-results}

In this section, we define the elements $s_i^\bw \in \mathcal A$ and the vectors $\lambda_i^\bw$ to present the main theorem of this paper precisely. The key idea is that we make the formulae \eqref{formulae} inductively hold for each mutation sequence $\bw$. This process shows that there is a unique term in $s_i^\bw$ that survives mod $2 \mathcal A$ without regard to the choice of an ordering $\prec$. More precisely, we prove $s_i^\bw  \equiv r_i^\bw$ (mod $2 \mathcal A$). When $B$ is acyclic, the $c$-vectors $c_i^\bw$ are the reflection vectors of $\pi(r_i^\bw)$ as shown in \cite{ST} with the linear ordering $\prec$ defined by $i \prec j$ if and only if $b_{ij}<0$. However, for general $B$, it is not true any more and comparing $r_i^\bw$ with $s_i^\bw$ will help us understand how the reflections $r_i^\bw$ arise in relation to the $c$-vectors $c_i^\bw$ as it will be shown as a part of the main theorem that $c_i^\bw=\lambda_i^\bw$. 

\medskip

Throughout this section, assume that $B=[b_{ij}]$ is a skew-symmetrizable matrix. Fix a linear ordering $\prec$ on $\mathcal I$ to obtain its associated GIM $A=[a_{ij}]$ from \eqref{eqn-gim}.

\begin{LExa}\label{exa-part1}
{\it As a running example in this section, we consider the skew-symmetrizable matrix 
$$B = \begin{bmatrix}
0 & 1 & -3 \\
-2 & 0 & -2 \\
3 & 1 & 0
\end{bmatrix}$$ 
with symmetrizer $D = \mathrm{diag}(1,2,1)$ and linear ordering $1 \prec 2 \prec 3.$
Following the convention in \eqref{eqn-gim}, we produce the GIM 
$$ A = \begin{bmatrix}
2 & 1 & -3 \\
2 & 2 & -2 \\
-3 & -1 & 2
\end{bmatrix}.$$}
\end{LExa}

Assume that a mutation sequence $\bw$ is given. We will inductively define the elements $s_i^\bw \in \mathcal A$ and the vectors $\lambda^\bw_i$, $i \in \mathcal I$, in what follows. The procedure is summarized in Table \ref{tab-1}.

\medskip
\begin{table}[t] 
\begin{center}
\begin{tikzpicture}[scale=0.25mm]
\draw [help lines] (0,9)--(0,10)--(6,10)--(6,9)--(0,9);
\node at (3,9.5) {Define $\lambda_i=\lambda_i^{[\, ]}$};
\draw[->] (3,9)--(3,8); 
\node [left] at (3,8.5) {\tiny Def $\Lambda$-\ref{L:1}};
\draw [help lines] (0,7)--(0,8)--(6,8)--(6,7)--(0,7);
\node at (3,7.5) {\small Define $s_i=s_i^{[\, ]}$ and $e_i=e_i^{[\, ]}$};
\draw[->] (3,7)--(3,6);
\node [left] at (3,6.5) {\tiny Def $\Lambda$-\ref{L:2}};
\draw [help lines] (0,5)--(0,6)--(6,6)--(6,5)--(0,5);
\node at (3,5.5) {\tiny Define $\mathcal P_s([\,], [i_1])$ and $\mathcal P_\tau([\,], [i_1])$};
\draw[->] (3,5)--(3,4);
\node [left] at (3,4.5) {\tiny Def $\Lambda$-\ref{L:3}};
\draw [help lines] (0,3)--(0,4)--(6,4)--(6,3)--(0,3);
\node at (3,3.5) {Define $\tau_i=\tau_i^{[\, ]}$};
\draw[->] (3,3)--(3,2);
\node [left] at (3,2.5) {\tiny Def $\Lambda$-\ref{L:4}};
\draw [help lines] (0,1)--(0,2)--(6,2)--(6,1)--(0,1);
\node at (3,1.5) {Define $\lambda_i^{[i_1]}$};
\draw[->] (3,1)--(3,0);
\draw [help lines] (0,-1)--(0,0)--(6,0)--(6,-1)--(0,-1);
\node at (3,-0.5) {Let $\bv=[\,]$ and $m=1$};
\draw[->] (6,-0.5)--(7.5,-0.5)--(7.5,9.5)--(9,9.5);
\draw [help lines] (9,9)--(9,10)--(18,10)--(18,9)--(9,9);
\node [left] at (8.5,10) {\tiny Def $\Lambda$-\ref{L:5}};
\node [left] at (20,10) {\tiny Def $\Lambda$-\ref{L:5}};
\node [left] at (13.5,8.5) {\tiny Def $\Lambda$-\ref{L:6}};
\node [left] at (13.5,6.5) {\tiny Def $\Lambda$-\ref{L:7}};
\node [left] at (13.5,4.5) {\tiny Def $\Lambda$-\ref{L:8}};
\node [left] at (13.5,2.5) {\tiny Def $\Lambda$-\ref{L:9}};
\node at (13.5,9.5) {Define $e_i^{\bv[i_m]}$};
\draw[->] (13.5,9)--(13.5,8);
\draw [help lines] (9,7)--(9,8)--(18,8)--(18,7)--(9,7);
\node at (13.5,7.5) {Define $s_i^{\bv[i_m]}$};
\draw[->] (13.5,7)--(13.5,6);
\draw [help lines] (8.4,5)--(8.4,6)--(18.6,6)--(18.6,5)--(8.4,5);
\node at (13.5,5.5) {\tiny Define $\mathcal P_s(\bv[i_m], \bv[i_m,i_{m+1}])$ and $\mathcal P_\tau(\bv[i_m], \bv[i_m,i_{m+1}])$};
\draw[->] (13.5,5)--(13.5,4);
\draw [help lines] (9,3)--(9,4)--(18,4)--(18,3)--(9,3);
\node at (13.5,3.5) {Define $\tau_i^{\bv[i_m]}$};
\draw[->] (13.5,3)--(13.5,2);
\draw [help lines] (9,1)--(9,2)--(18,2)--(18,1)--(9,1);
\node at (13.5,1.5) {Define $\lambda_i^{\bv[i_m,i_{m+1}]}$};
\draw [help lines] (9,-1)--(9,0)--(18,0)--(18,-1)--(9,-1);
\node at (13.5,-0.5) {$\bv[i_m]\mapsto \bv$ and $m+1\mapsto m$};
\draw[->] (13.5,1)--(13.5,0);
\draw[->] (18,-0.5)--(19.5,-0.5)--(19.5,9.5)--(18,9.5);
\end{tikzpicture}
\end{center}
\caption{Flow chart for defining $s_i^\bw$ and $\lambda_i^\bw$}
\label{tab-1}
\end{table}

\medskip

For convenience, we recall the definition of $\mathcal A$ and its representation on $\mathbb Z^n$. As before, set $\lambda_1=(1,0,\dots , 0), \ \lambda_2=(0,1,0, \dots, 0), \dots, \ \lambda_n=(0, \dots, 0, 1)$ to be a basis of $\mathbb Z^n$.
\begin{LDef}\label{L:1}
Let $\mathcal A$ be the (unital) $\mathbb Z$-algebra generated by $s_i, e_i$, $i \in \mathcal I$, subject to the following relations:
$$ s_i^2=1, \quad \sum_{i=1}^n e_i =1, \quad s_ie_i = -e_i, \quad e_is_j=  \begin{cases} s_i+e_i-1 &\text{if } i =j, \\ e_i &\text{if } i \neq j, \end{cases}  \quad e_ie_j=  \begin{cases} e_i &\text{if } i =j, \\ 0 & \text{if } i \neq j. \end{cases}$$
Define a representation $\pi:\mathcal A \rightarrow \mathrm{End}(\mathbb Z^n)$ by 
\begin{equation} 
\pi(s_i)(\lambda_j) = \lambda_j-a_{ji} \lambda_i \quad \text{ and } \quad \pi(e_i)(\lambda_j)  = \delta_{ij}\lambda_i \qquad \text{ for } i,j \in \mathcal I, 
\end{equation} and by extending it through linearity, 
where $\delta_{ij}$ is the Kronecker delta. We will suppress $\pi$ when we write the action of an element of $\mathcal A$ on $\mathbb Z^n$. 
\end{LDef}

\begin{LExa}\label{exa-part2}{\it
Continuing from Example $\Lambda$-\ref{exa-part1}, the action of $s_i$, $i=1,2,3$, are respectively given by the following matrices: $$
\begin{bmatrix}
-1 & 0 & 0 \\
-2 & 1 & 0 \\
3 & 0 & 1
\end{bmatrix}, \quad
\begin{bmatrix}
1 & -1 & 0 \\
0 & -1 & 0 \\
0 & 1 & 1
\end{bmatrix}, \quad
\begin{bmatrix}
1 & 0 & 3 \\
0 & 1 & 2 \\
0 & 0 & -1
\end{bmatrix}.$$
Here the action of $s_i$ on the vector $\lambda_j$ is to be understood by multiplication of the matrix on the right.}  
\end{LExa}

\begin{LDef}\label{L:2}
Suppose that $\bw$ starts with $k$. 
Let $\mathcal P_s([\,], [k])$ be the set of $(i,j)$, $i,j \in \mathcal I$, such that 
\begin{enumerate}
\item[] $\lambda_i > s_k(\lambda_i) \text{ and } \lambda_j < s_k(\lambda_j) \text{ and } ( k \prec  i \prec j \text{ or } i \prec j \prec k)$, or 
\item[] $\lambda_j < s_k(\lambda_j) \text{ and } k=i \prec j$. 
\end{enumerate} 
Let $\mathcal P_\tau([\,], [k])$ be the set of $(i,j)$, $i,j \in \mathcal I$, such that 
\begin{enumerate}
\item[] $\lambda_i > s_k(\lambda_i) \text{ and } \lambda_j < s_k(\lambda_j) \text{ and } ( k \prec  i \prec j \text{ or } i \prec j \prec k)$, or 
\item[] $\lambda_j > s_k(\lambda_j) \text{ and } k=i \succ j$. 
\end{enumerate}
\end{LDef}
\begin{LDef}\label{L:3}
Define \[ e_{\tau,i} = \sum e_j \in \mathcal A,\]
where the sum is over $j$ such that $(i,j) \in \mathcal P_\tau([\,],[k])$ or $(j,i) \in \mathcal P_\tau([\,],[k])$, and 
define 
\begin{equation} \label{tau_i} \tau_i = s_i+ 2  (1-s_i) e_{\tau,i} \qquad \text{ for } i \in \mathcal I. \end{equation}
\end{LDef}
\begin{LDef} \label{L:4}
Define 
\begin{equation}\label{eqn-lambda1}
 \lambda_i^{[k]} =  \begin{cases}  \tau_k (\lambda_i) & \text{ if } \lambda_i < s_k(\lambda_i) \text{ and } k \prec i, \text{ or if } \lambda_i > s_k(\lambda_i) \text{ and } k \succ i, \text{ or if } i=k, \\  \lambda_i & \text{ otherwise}. \end{cases} 
\end{equation}
\end{LDef}

\begin{LExa}\label{exa-part3}
{\it Continuing from Example $\Lambda$-\ref{exa-part2}, take $\bw = [2,3]$ so $k=2.$ We have $\mathcal{P}_s([],[2]) = \{(2,3)\}$ and $\mathcal{P}_\tau([],[2]) = \{(2,1)\}$. It follows that $e_{\tau,1} = e_2,$ $e_{\tau,2} = e_1,$ and $e_{\tau,3} = 0.$ Putting everything together we see that $$\tau_1 = s_1+2(1-s_1)e_2, \quad 
\tau_2 = s_2+2(1-s_2)e_1, \quad 
\tau_3 = s_3.$$
We then have 
 $$ \tau_2(\lambda_1) = (2-s_2)(\lambda_1) = (1, 1, 0), \quad
 \tau_2(\lambda_2) = s_2(\lambda_2) =(0, -1, 0), \quad
\tau_2(\lambda_3)= s_2(\lambda_3) = (0, 1, 1).
 $$
By \eqref{eqn-lambda1} we define  $\lambda_i^{[2]} := \tau_2(\lambda_i)$ for all $i \in\mathcal{I}$.}
\end{LExa}

\begin{LDef}\label{L:5}
Inductively, assume $\bw=\bv [k, \ell, \dots, m]$, including the case $\bv=[\,]$. 
For $i\neq k$, define
\begin{align} \label{eibvk} 
e_i^{\bv[k]} &=   \begin{cases}  \tau_k^\bv e_i^\bv \tau_k^\bv & \text{ if } \lambda_i^\bv < s_k^\bv(\lambda_i^\bv) \text{ and } k \prec i, \text{ or if } \lambda_i^\bv > s_k^\bv(\lambda_i^\bv) \text{ and } k \succ i,  \\   e_i^\bv & \text{ otherwise}, \end{cases} 
\end{align}
and
\[ e_k^{\bv[k]} =e_k^\bv- e_k^\bv e_+^{\bv[k]}, \]
where we set \[ e_+^{\bv[k]} = \sum_{ j \neq k,\  \lambda_j^{\bv[k]} \neq \lambda_j^\bv } e_j^{\bv[k]}. \] 
\end{LDef}

\begin{LExa}\label{exa-part5}
{\it Continuing from Example $\Lambda$-\ref{exa-part3} we have $k = 2, \ell = 3,$ and $\bv = [].$
For $i = 1,3$ we have $e_i^{[2]} = \tau_2e_i\tau_2.$ More explicitly, 
$$e_1^{[2]} = \tau_2e_1\tau_2 = (2-s_2)e_1, \qquad e_3^{[2]} = \tau_2e_3\tau_2 = s_2e_3.$$
For $i=2$, 
$$e_+^{[2]} = e_1^{[2]} + e_3^{[2]}  = 2e_1 - s_2(e_1-e_3)$$ and finally
$$e_2^{[2]} = e_2(1-e_1^{[2]} - e_3^{[2]}) = s_2(e_1-e_3)-e_1+e_2+e_3.$$}
\end{LExa}

\begin{LDef}\label{L:6}
Define \[ e^{\bv[k]}_{s,i} = \sum e_j^{\bv[k]} ,\]
where the sum is over $j$ such that $(i,j) \in \mathcal P_s(\bv,\bv[k])$ or $(j,i) \in \mathcal P_s(\bv,\bv[k])$, and  
define
\begin{align} \label{s_i^} s_i^{\bv[k]} &= \begin{cases} \tau_k^\bv \tau_i^\bv \tau_k^\bv + 2 ( 1- \tau_k^\bv \tau_i^\bv \tau_k^\bv ) e^{\bv[k]}_{s,i} &  \text{ if } \lambda_i^\bv < s_k^\bv(\lambda_i^\bv) \text{ and } k \prec i, \text{ or if } \lambda_i^\bv > s_k^\bv(\lambda_i^\bv) \text{ and } k \succ i, \\ \tau_i^\bv+ 2(1-\tau_i^\bv) e^{\bv[k]}_{s,i} &  \text{ otherwise}.
\end{cases} 
\end{align}
\end{LDef}

\begin{LExa}\label{exa-part6}
{\it In Example $\Lambda$-\ref{exa-part3} we computed $\mathcal{P}_s([],[2]) = \{(2,3)\}$ so 
$$e^{[2]}_{s,1} = 0, \qquad e^{[2]}_{s,2} = e_3^{[2]},\qquad e^{[2]}_{s,3} =e_2^{[2]}.$$
Now by comparing $s_i(\lambda_i)$ given in Example $\Lambda$-\ref{exa-part2} to $\lambda_i$, we have
\begin{align*}
s_1^{[2]}&= \tau_2\tau_1\tau_2 + 2(1-\tau_2\tau_1\tau_2)e_{s,1}^{[2]}= \tau_2\tau_1\tau_2\\
&= (2-2s_1+s_2s_1)s_2+ 2(1-2s_2+2s_1s_2-s_2s_1s_2)e_1+2(-2+2s_1+2s_2-s_2s_1)e_3,\\
s_2^{[2]}&=\tau_2 + 2(1-\tau_2)e_{s,2}^{[2]}=2(e_1-e_3)+s_2(1-2(e_1-e_3))=s_2+2(1-s_2)(e_1-e_3),\\
s_3^{[2]}&= \tau_3 + 2(1-\tau_3)e_{s,3}^{[2]} = s_2s_3s_2+2(1+s_2s_3)e_2+2(1-2s_2-s_2s_3s_2)e_3.
\end{align*}  }
\end{LExa}

\begin{LDef}\label{L:7}
Let $\mathcal P_s(\bv[k], \bv[k,\ell])$ be the collection of $(i,j)$ such that 
\begin{enumerate}
\item[] $( \ell \prec i \prec j \text{ or } i \prec j \prec \ell)  \text{ and }  \lambda_i^{\bv[k]} > s_\ell^{\bv[k]}(\lambda_i^{\bv[k]}) \text{ and } \lambda_j^{\bv[k]} < s_\ell^{\bv[k]}(\lambda_j^{\bv[k]})  $, or  
\item[] $\ell=i \succ j  \text{ and } \lambda_\ell^{\bv[k]} < 0  \text{ and }    \lambda_j^{\bv[k]} > s_\ell^{\bv[k]}(\lambda_j^{\bv[k]})$, or 
\item[] $\ell=i \prec j  \text{ and }  \lambda_\ell^{\bv[k]} > 0 \text{ and }     \lambda_j^{\bv[k]} < s_\ell^{\bv[k]}(\lambda_j^{\bv[k]})$. 
\end{enumerate}
Similarly, let $\mathcal P_\tau(\bv[k], \bv[k,\ell])$ be the collection of $(i,j)$ such that 
\begin{enumerate}
\item[] $( \ell \prec i \prec j \text{ or } i \prec j \prec \ell)  \text{ and }  \lambda_i^{\bv[k]} > s_\ell^{\bv[k]}(\lambda_i^{\bv[k]}) \text{ and } \lambda_j^{\bv[k]} < s_\ell^{\bv[k]}(\lambda_j^{\bv[k]})  $, or  
\item[] $\ell=i \succ j  \text{ and } \lambda_\ell^{\bv[k]} >0  \text{ and }    \lambda_j^{\bv[k]} > s_\ell^{\bv[k]}(\lambda_j^{\bv[k]})$, or 
\item[] $\ell=i \prec j  \text{ and }  \lambda_\ell^{\bv[k]} <0 \text{ and }     \lambda_j^{\bv[k]} < s_\ell^{\bv[k]}(\lambda_j^{\bv[k]})$. 
\end{enumerate}
\end{LDef}

\begin{LExa}\label{exa-part7}
{\it Continuing from Example $\Lambda$-\ref{exa-part6} we have $$
s^{[2]}_3(\lambda^{[2]}_1) = (1, 1, 1),\quad
s^{[2]}_3(\lambda^{[2]}_2) = (0, 3, 2), \quad
s^{[2]}_3(\lambda^{[2]}_3) = (0, -4, -3)$$
so $\mathcal P_s([2],[2,3]) = \emptyset$ and $\mathcal P_\tau([2],[2,3]) = \{(3,2)\}.$ }
\end{LExa}

\begin{LDef}\label{L:8}
Define \[ e^{\bv[k]}_{\tau,i} = \sum e_j^{\bv[k]} \in \mathcal A,\]
where the sum is over $j$ such that $(i,j) \in \mathcal P_\tau(\bv[k],\bv[k,\ell])$ or $(j,i) \in \mathcal P_\tau(\bv[k],\bv[k,\ell])$, and define 
\begin{equation} \label{tau_i^} \tau_i^{\bv[k]} = s_i^{\bv[k]}+ 2  (1-s_i^{\bv[k]}) e^{\bv[k]}_{\tau,i} \qquad \text{ for } i \in \mathcal I. \end{equation}
\end{LDef}
\begin{LDef} \label{L:9}
Finally, define 
\begin{equation} \label{fin-la} \lambda_j^{\bv[k,\ell]} =  \begin{cases}  \tau_\ell^{\bv[k]} ( \lambda_j^{\bv[k]} ) & \text{ if } \lambda_j^{\bv[k]} < s_\ell^{\bv[k]}(\lambda_j^{\bv[k]}) \text{ and } \ell \prec j, \\ & \qquad \text{ or if } \lambda_j^{\bv[k]} > s_\ell^{\bv[k]}(\lambda_j^{\bv[k]})  \text{ and } \ell \succ j,  \text{ or if } \ell=j,  \\ \lambda_j^{\bv[k]} & \text{ otherwise}. \end{cases} \end{equation}
\end{LDef}

\begin{LExa}\label{exa-part8}
{\it Continuing from Example $\Lambda$-\ref{exa-part7} we have $$ e_{\tau,1}^{[2,3]} = 0,\quad e_{\tau,2}^{[2,3]} = e_3^{[2]},\quad e_{\tau,3}^{[2,3]} = e_2^{[2]}.$$
Furthermore, $$\tau_1^{[2]} = s_1^{[2]},\quad\tau_2^{[2]} = s_{2}-2(1-s_{2} )e_{1},\quad \tau_3^{[2]} = s_{2}s_{3}s_{2} +2(1-s_{2}s_{3}s_{2} + s_{2}s_{3} - s_{2})e_{1}.$$
In Example $\Lambda$-\ref{exa-part6} we computed $s_3^{[2]}(\lambda_i^{[2]}).$ Finishing our running example we conclude that 
\begin{align*} \lambda_1^{[2,3]}& =  \lambda_1^{[2]}  = (2-s_{2})(\lambda_{1}) = (1,1,0), \\
\lambda_2^{[2,3]} &= \tau_3^{[2]}(\lambda_2^{[2]}) = s_{2}s_{3}(\lambda_{2}) = (0, 1, 2), \\
\lambda_3^{[2,3]} &= \tau_3^{[2]}(\lambda_3^{[2]})  = -s_{2}(\lambda_{3}) =(0, -1, -1). 
\end{align*} }
\end{LExa}

For any mutation sequence $\bw$, set \[ \Lambda^\bw= \begin{bmatrix} \lambda_1^{\bw } \\ \vdots \\  \lambda_n^{\bw } \end{bmatrix}. \] 

Now we restate the main theorem of this paper.

\begin{Thm}[Theorem \ref{vague_conj-1}] \label{con-prec-1}
Let $B$ be a skew-symmetrizable matrix. 
Fix a linear ordering $\prec$ on $\mathcal I$ to obtain a GIM $A$. Then,   
for any mutation sequence $\bw$, we have 
\begin{equation*} \lambda_i^\bw =  c_i^\bw \qquad \text{ for all $i \in \mathcal I$}, \tag{\it C1} \end{equation*}
or equivalently, 
\[ \Lambda^\bw = C^\bw ;\]
for $i,j \in \mathcal I$,
\begin{align*} s_i^\bw(\lambda_j^\bw)& = \begin{cases} \lambda_j^\bw +b_{ji}^\bw \lambda_i^\bw & \text{ if } i \prec j, \\ -\lambda_j^\bw & \text{ if } i =j, \\\lambda_j^\bw -b_{ji}^\bw \lambda_i^\bw & \text{ if } i \succ j , \tag{\it C2} \end{cases} 
& e_i^\bw(\lambda_j^\bw)& = \delta_{ij} \lambda_j^\bw ;
\end{align*}
moreover, for all $i \in I$, 
\[\tag{\it C3}  s_i^\bw \equiv r_i^\bw \quad (\mathrm{mod}\ 2 \mathcal A) . \]
\end{Thm}

In what follows, we prove (C3). A proof of (C1) and (C2) will be given in Section \ref{sec-proof-c1c2}.

\begin{proof}[Proof of (C3)]
Notice from \eqref{tau_i^} that $s_i^\bw \equiv \tau_i^\bw$ modulo $2 \mathcal A$. Then the equation \eqref{s_i^} becomes modulo $2 \mathcal A$ \begin{align}  s_i^{\bv[k]} &\equiv \begin{cases} s_k^\bv s_i^\bv s_k^\bv &  \text{ if } \lambda_i^\bv < s_k^\bv(\lambda_i^\bv) \text{ and } k \prec i, \text{ or if } \lambda_i^\bv > s_k^\bv(\lambda_i^\bv) \text{ and } k \succ i, \\ s_i^\bv &  \text{ otherwise}.
\end{cases} 
\end{align}
Using (C1) and (C2), both of the conditions $\lambda_i^\bv < s_k^\bv(\lambda_i^\bv),  k \prec i$ and $\lambda_i^\bv > s_k^\bv(\lambda_i^\bv), k \succ i$ can be rewritten as
\[ b_{ik}^\bv \lambda_k^\bv=b_{ik}^\bv c_k^\bv >0, \]
which does not depend on the choice of a GIM.
 Now (C3) follows from the definitions \eqref{tau_i}, \eqref{s_i^} and 
\eqref{tau_i^} and from induction.
\end{proof}

\subsection{Some observations}
We close this section with examples which show some relationship between 
$c$-vectors and L\"osungen. 

\begin{Exa}\label{ex-counterexa2}
Consider the matrix $B = \begin{bmatrix}
0 & -1 & -1 & -1 \\
1 & 0 & 1 & -1 \\
1 & -1 & 0 & 1 \\
1 & 1 & -1 & 0
\end{bmatrix}$. 
The mutation sequence $[1,2,3,4,2]$ produces the $c$-vector $(5,2,2,2)$ which is not a L\"osung for any choice of GIM associated to $B$. 
\end{Exa}

Example~\ref{ex-counterexa3} below shows that even if a $c$-vector is a real L\"osung our formula may not always express it as such. 

\begin{Exa}\label{ex-counterexa3}
Consider the matrix $B = \begin{bmatrix}
0 & 1 & 0 & 0 \\
-1 & 0 & -1 & 0 \\
0 & 1 & 0 & 1 \\
0 & 0 & -1 & 0
\end{bmatrix}.$ This is a finite-type matrix that corresponds to an orientation of the Dynkin diagram $A_4.$ After the mutation sequence $\bw = [2,4,2]$ with the GIM associated to the linear order $4 \prec 2 \prec 3 \prec 1$ our formula produces $$\lambda_3^\bw = -s_{2}s_{4}s_{2}\lambda_{3} - 2s_{2}\lambda_{3} + 2\lambda_{3} + 2s_{4}s_{2}\lambda_{3} = (0, 0, 1, 1).$$ However, we also have $s_{2}s_{4}s_{2}\lambda_{3} = (0,0,1,1)$ so we see that $\lambda_3^\bw$ could just be expressed as the real L\"osung $s_{2}s_{4}s_{2}\lambda_{3}$ as opposed to the linear combination of real L\"osungen given above. For completeness, we have $s_{2}\lambda_{3} = (0,1,1,0)$ and $s_{4}s_{2}\lambda_{3} = (0,1,1,1).$

It is also worth noting that the matrix representation of $-s_{2}s_{4}s_{2} - 2s_{2} + 2 + 2s_{4}s_{2}$ is not equal to the matrix representation of $s_{2}s_{4}s_{2}$. Furthermore, for any choice of linear ordering the expression for $\lambda_3^\bw$ that our formula produces will always have three or four terms even though the vector is a real L\"osung.
\end{Exa}

\section{Proof of (C1) and (C2) in Theorem \ref{con-prec-1}} \label{sec-proof-c1c2}
In this section we prove Theorem \ref{con-prec-1}.
We start with the following proposition which shows that $s_i^\bw, e_i^\bw$ satisfy natural relations for each $\bw$. 
\begin{Prop}
For $i, j \in \mathcal I$ and for any mutation sequence $\bw$, the following relations hold:
\begin{align}
& \sum_{i=1}^n e_i^\bw =1, \label{vava-1} \\
&  e_i^\bw e_j^\bw=  \delta_{ij} e_i^\bw , \label{vava-2}\\
& e_i^\bw s_j^\bw=  \begin{cases} s_i^\bw+e_i^\bw-1 &\text{if } i =j, \\ e_i^\bw &\text{if } i \neq j, \end{cases} \label{vava-3} \\
& e_i^\bw \tau_j^\bw=  \begin{cases} \tau_i^\bw+e_i^\bw-1 &\text{if } i =j, \\ e_i^\bw &\text{if } i \neq j, \end{cases} \label{vava-4} \\
& s_i^\bw s_i^\bw=1, \qquad \quad \tau_i^\bw \tau_i^\bw=1,\label{vava-5} \\ 
&  s_i^\bw e_i^\bw = -e_i^\bw, \qquad  \tau_i^\bw e_i^\bw = -e_i^\bw .  \label{vava-6} 
\end{align}

\end{Prop}

\begin{proof}

We use induction. If $\bw =[]$, all the relations follow from the definitions. Assume the relations hold for $\bv$. In what follows, we show that they hold for $\bv[k]$, $k \in \mathcal I$.  
 
\underline{Relation \eqref{vava-1}:}
Since $e_i^\bv \tau_k^\bv = e_i^\bv$ for $i \neq k$ by induction, we have $(1-e_k^\bv)\tau_k^\bv=(1-e_k^\bv)$, and obtain
\begin{align*}
\sum_{i=1}^n e_i^{\bv[k]} &= e_k^{\bv[k]} + \sum_{i \neq k} e_i^{\bv[k]}\\
&= e_k^\bv -e_k^\bv \sum_{\substack{ i \neq k \\ \lambda_i^{\bv[k]} \neq \lambda_i^\bv}} e_i^{\bv[k]} + \sum_{\substack{ i \neq k \\ \lambda_i^{\bv[k]} \neq \lambda_i^\bv}} e_i^{\bv[k]} + \sum_{\substack{ i \neq k \\ \lambda_i^{\bv[k]} =\lambda_i^\bv}} e_i^{\bv[k]}\\
&= e_k^\bv+(1 -e_k^\bv) \sum_{\substack{ i \neq k \\ \lambda_i^{\bv[k]} \neq \lambda_i^\bv}} \tau_k^\bv e_i^\bv \tau_k^\bv + \sum_{\substack{ i \neq k \\ \lambda_i^{\bv[k]} =\lambda_i^\bv}} e_i^\bv\\
&= e_k^\bv+\sum_{\substack{ i \neq k \\ \lambda_i^{\bv[k]} \neq \lambda_i^\bv}} (1 -e_k^\bv) e_i^\bv+ \sum_{\substack{ i \neq k \\ \lambda_i^{\bv[k]} =\lambda_i^\bv}} e_i^\bv=\sum_{i=1}^n e_i^\bv =1.    
\end{align*}

\underline{Relations \eqref{vava-2}:}
Suppose that $i \neq k$ and $j \neq k$. Note that $e_i^\bv \tau_k^\bv= e_i^\bv$ and $e_j^\bv \tau_k^\bv= e_j^\bv$. 
Assume $e_i^{\bv[k]}=e_i^\bv$ and $e_j^{\bv[k]}=e_j^\bv$. Then
\[ e_i^{\bv[k]} e_j^{\bv[k]} =  e_i^\bv e_j^\bv = \delta_{ij} e_i^\bv =\delta_{ij} e_i^{\bv [k]} . \]
Assume $e_i^{\bv[k]}=\tau_k^\bv e_i^\bv \tau_k^\bv$ and $e_j^{\bv[k]}=e_j^\bv$. Then
\[ e_i^{\bv[k]} e_j^{\bv[k]} = \tau_k^\bv e_i^\bv \tau_k^\bv e_j^\bv = \tau_k^\bv e_i^\bv e_j^\bv = \delta_{ij} \tau_k^\bv e_i^\bv =\delta_{ij} e_i^{\bv [k]} . \]
Assume $e_i^{\bv[k]}=e_i^\bv$ and $e_j^{\bv[k]}=\tau_k^\bv e_j^\bv \tau_k^\bv$. Then
\[ e_i^{\bv[k]} e_j^{\bv[k]} =  e_i^\bv \tau_k^\bv e_j^\bv \tau_k^\bv = e_i^\bv e_j^\bv \tau_k^\bv = \delta_{ij} e_i^\bv
 \tau_k^\bv  =\delta_{ij} e_i^{\bv [k]} . \]
Assume $e_i^{\bv[k]}=\tau_k^\bv e_i^\bv \tau_k^\bv$ and $e_j^{\bv[k]}=\tau_k^\bv e_j^\bv \tau_i^\bv$. Then
\[ e_i^{\bv[k]} e_j^{\bv[k]} =  \tau_k^\bv e_i^\bv \tau_k^\bv \tau_k^\bv e_j^\bv \tau_k^\bv = \tau_k^\bv e_i^\bv e_j^\bv \tau_k^\bv =\delta_{ij} \tau_k^\bv e_i^\bv \tau_k^\bv =\delta_{ij} e_i^{\bv [k]} . \]   			

For $i \neq k$ and $j \neq k$, write $A= \left ( 1- \displaystyle{\sum_{i \neq k, \ \lambda_i^{\bv[k]} \neq \lambda_i^\bv}} e_i^{\bv[k]} \right)$ for the time being, and we get
\begin{align*} e_k^{\bv[k]} e_j^{\bv[k]} &= e_k^\bv  A e_j^{\bv[k]} =\begin{cases} e_k^\bv ( e_j^{\bv[k]} - e_j^{\bv[k]} ) =0 & \text{ if } \lambda_i^{\bv[k]} \neq \lambda_i^\bv, \\  e_k^\bv e_j^{\bv[k]}=e_k^\bv e_j^\bv=0 & \text{ if } \lambda_i^{\bv[k]} = \lambda_i^\bv, \end{cases} \\  
e_i^{\bv[k]} e_k^{\bv[k]} &= e_i^{\bv[k]} e_k^\bv A 
= \begin{cases} \tau_k^\bv e_i^\bv \tau_k^\bv e_k^\bv A = \tau_k^\bv e_i^\bv e_k^\bv A=0 & \text{ if } \lambda_i^{\bv[k]} \neq \lambda_i^\bv ,\\ e_i^\bv e_k^\bv A =0 & \text{ if } \lambda_i^{\bv[k]} = \lambda_i^\bv, \\ \end{cases} \\
e_k^{\bv[k]} e_k^{\bv[k]} &= e_k^\bv A e_k^\bv A = (e_k^\bv - \sum_{i \neq k, \ \lambda_i^{\bv[k]} \neq \lambda_i^\bv} e_k^\bv \tau_k^\bv e_i^\bv \tau_k^\bv e_k^\bv) A= e_k^\bv A = e_k^{\bv[k]}. 
\end{align*}
We have proven \[ e_i^{\bv[k]} e_j^{\bv[k]}= \delta_{ij} e_i^{\bv[k]} \] for all $i,j \in \mathcal I$.

\underline{Relations \eqref{vava-3}:}
Assume that $i \neq j$ and $i \neq k$. Suppose that $e_i^{\bv[k]}=e_i^\bv$ and $e_j^{\bv[k]}=e_j^\bv$. Then we have
\begin{align*}
e_i^{\bv[k]} s_j^{\bv[k]} &= e_i^\bv ( \tau_j^\bv + 2(1-\tau_j^\bv) e_{s,j}^{\bv[k]}) = e_i^\bv + 2 e_i^\bv (1-\tau_j^\bv) e_{s,j}^{\bv[k]}=e_i^\bv=e_i^{\bv[k]}.
\end{align*}
Suppose that $e_i^{\bv[k]}=\tau_k^\bv e_i^\bv \tau_k^\bv$ and $e_j^{\bv[k]}=e_j^\bv$.
\begin{align*}
e_i^{\bv[k]} s_j^{\bv[k]} &= \tau_k^\bv e_i^\bv \tau_k^\bv ( \tau_j^\bv + 2(1-\tau_j^\bv) e_{s,j}^{\bv[k]}) = \tau_k^\bv e_i^\bv + 2  (\tau_k^\bv e_i^\bv -\tau_k^\bv e_i^\bv \tau_j^\bv) e_{s,j}^{\bv[k]}=\tau_k^\bv e_i^\bv=e_i^{\bv[k]}.
\end{align*}
Suppose that $e_i^{\bv[k]}=e_i^\bv$ and $e_j^{\bv[k]}=\tau_k^\bv e_j^\bv \tau_k^\bv$.
\begin{align*}
e_i^{\bv[k]} s_j^{\bv[k]} &= e_i^\bv \left [ \tau_k^\bv \tau_j^\bv \tau_k^\bv + 2(1- \tau_k^\bv \tau_j^\bv \tau_k^\bv) e_{s,j}^{\bv [k]} \right ] = e_i^\bv + 2 e_i^\bv (1-\tau_k^\bv \tau_j^\bv \tau_k^\bv) e_{s,j}^{\bv[k]}=e_i^\bv=e_i^{\bv[k]}.
\end{align*}
Suppose that $e_i^{\bv[k]}=\tau_k^\bv e_i^\bv \tau_k^\bv$ and $e_j^{\bv[k]}=\tau_k^\bv e_j^\bv \tau_k^\bv$. Note that 
\[ \tau_k^\bv e_i^\bv \tau_k^\bv  \tau_k^\bv \tau_j^\bv \tau_k^\bv =\tau_k^\bv e_i^\bv \tau_j^\bv \tau_k^\bv= \tau_k^\bv e_i^\bv \tau_k^\bv.\]  Then we have
\begin{align*}
e_i^{\bv[k]} s_j^{\bv[k]} &=\tau_k^\bv e_i^\bv \tau_k^\bv \left [ \tau_k^\bv \tau_j^\bv \tau_k^\bv + 2(1- \tau_k^\bv \tau_j^\bv \tau_k^\bv) e_{s,j}^{\bv [k]} \right ] \\
&= \tau_k^\bv e_i^\bv \tau_k^\bv + 2(\tau_k^\bv e_i^\bv \tau_k^\bv-\tau_k^\bv e_i^\bv \tau_k^\bv)e_{s,j}^{\bv[k]} = e_i^{\bv[k]}.
\end{align*}

Assume that $i =k \neq j$.
Suppose that $e_j^{\bv[k]} = e_j^\bv$. Note that 
\[e_k^{\bv[k]} \tau_j^\bv=\left ( e_k^\bv- \displaystyle{\sum_{\ell \neq k, \ \lambda_\ell^{\bv[k]} \neq \lambda_\ell^\bv}} e_k^\bv \tau_k^\bv e_\ell^\bv \tau_k^\bv \right) \tau_j^\bv =\left ( e_k^\bv- \displaystyle{\sum_{\ell \neq k, \ \lambda_\ell^{\bv[k]} \neq \lambda_\ell^\bv}} e_k^\bv \tau_k^\bv e_\ell^\bv \tau_k^\bv \right)=e_k^{\bv[k]}.\]
Then we have
\begin{align*}
e_k^{\bv[k]} s_j^{\bv[k]} &= e_k^{\bv[k]} ( \tau_j^\bv + 2(1-\tau_j^\bv) e_{s,j}^{\bv[k]})\\
&= e_k^{\bv[k]} \tau_j^\bv + 2 e_k^{\bv[k]}(1-\tau_j^\bv) e_{s,j}^{\bv[k]} 
= e_k^{\bv[k]}.
\end{align*}
Suppose that $e_j^{\bv[k]} = \tau_k^\bv e_j^\bv \tau_k^\bv$. Note that 
\begin{align*} e_k^{\bv[k]} \tau_k^\bv \tau_j^\bv \tau_k^\bv &=(1- \sum_{\ell \neq k} e_\ell^{\bv[k]}) \tau_k^\bv \tau_j^\bv \tau_k^\bv \\
&= \tau_k^\bv \tau_j^\bv \tau_k^\bv - \sum_{\ell \neq k, \ \lambda_\ell^{\bv[k]}=\lambda_\ell^\bv} e_\ell^\bv\tau_k^\bv \tau_j^\bv \tau_k^\bv-\sum_{\ell \neq k, \ \lambda_\ell^{\bv[k]} \neq \lambda_\ell^\bv} \tau_k^\bv e_\ell^\bv\tau_j^\bv  \tau_k^\bv \\
&= \tau_k^\bv \tau_j^\bv \tau_k^\bv - \tau_k^\bv e_j^\bv \tau_j^bv \tau_k^\bv - \sum_{\ell \neq k, \ \lambda_\ell^{\bv[k]}=\lambda_\ell^\bv} e_\ell^\bv-\sum_{\ell \neq k,j, \ \lambda_\ell^{\bv[k]} \neq \lambda_\ell^\bv} \tau_k^\bv e_\ell^\bv  \tau_k^\bv \\
&= 1- \tau_k^\bv e_j^\bv \tau_k^\bv -
\sum_{\ell \neq k, \ \lambda_\ell^{\bv[k]}=\lambda_\ell^\bv} e_\ell^\bv-\sum_{\ell \neq k,j, \ \lambda_\ell^{\bv[k]} \neq \lambda_\ell^\bv} \tau_k^\bv e_\ell^\bv  \tau_k^\bv \\
&= 1-\sum_{\ell \neq k} e_\ell^{\bv[k]} 
=e_k^{\bv[k]}.
\end{align*}
Then we have
\begin{align*}
e_k^{\bv[k]} s_j^{\bv[k]} &= e_k^{\bv[k]} (\tau_k^\bv \tau_j^\bv\tau_k^\bv + 2(1- \tau_k^\bv \tau_j^\bv \tau_k^\bv) e_{s,j}^{\bv[k]})= e_k^{\bv[k]}.
\end{align*}

Assume that $i=j \neq k$. Suppose that $\lambda_i^{\bv[k]}=\lambda_i^\bv$. Since $e_i^\bv e_{s,i}^{\bv[k]}=0$, we get 
\begin{align*}
e_i^{\bv[k]} s_i^{\bv[k]} &= e_i^\bv ( \tau_i^\bv + 2(1-\tau_i^\bv) e_{s,i}^{\bv[k]} ) \\
&= e_i^\bv \tau_i^\bv - 2 e_i^\bv \tau_i^\bv e_{s,i}^{\bv[k]} = \tau_i^\bv + e_i^\bv-1 -2 ( \tau_i^\bv + e_i^\bv -1) e_{s,i}^{\bv[k]} \\
&= \tau_i^\bv + 2(1-\tau_i^\bv) e_{s,i}^{\bv[k]} + e_i^\bv -1 
= s_i^{\bv[k]} + e_i^{\bv[k]} -1.
\end{align*} 
The case $\lambda_i^{\bv[k]} \neq \lambda_i^\bv$ is similar to the case  $\lambda_i^{\bv[k]}=\lambda_i^\bv$.
We omit the computations for this case.

Assume that $i=j=k$. Then 
\begin{align*}
e_k^{\bv[k]} s_k^{\bv[k]} &= (1 - \sum_{\ell \neq k} e_k^{\bv[k]} ) s_k^{\bv[k]} = s_k^{\bv[k]} - \sum_{\ell\neq k} e_\ell^{\bv[k]} s_k^{\bv[k]} = 
s_k^{\bv[k]} - \sum_{\ell\neq k} e_\ell^{\bv[k]}= s_k^{\bv[k]} + e_k^{\bv[k]} -1. \end{align*}

\underline{Relations \eqref{vava-4}:}
For $i \neq j$, we have $e_i^{\bv[k]}(1-s_i^{\bv[k]})=0$ and
\begin{align*}
e_i^{\bv[k]} \tau_j^{\bv[k]} &= e_i^{\bv[k]}(s_j^{\bv[k]} + 2(1 -s_j^{\bv[k]})e_{\tau,j}^{\bv[k]})=e_i^{\bv[k]}.
\end{align*}
For $i =j$, we get
\begin{align*}
e_i^{\bv[k]} \tau_i^{\bv[k]} &= e_i^{\bv[k]}(s_j^{\bv[k]} + 2(1 -s_i^{\bv[k]})e_{\tau,i}^{\bv[k]}) = e_i^{\bv[k]}s_j^{\bv[k]} + 2e_i^{\bv[k]}(1 -s_i^{\bv[k]})e_{\tau,i}^{\bv[k]} \\ 
&= s_i^{\bv[k]} + e_i^{\bv[k]} -1 -2 e_i^{\bv[k]} s_i^{\bv[k]} e_{\tau,i}^{\bv[k]} = s_i^{\bv[k]} + 2 (1-s_i^{\bv[k]}) e_{\tau,i}^{\bv[k]}+ e_i^{\bv[k]} -1 \\ &= \tau_i^{\bv[k]} + e_i^{\bv[k]} -1.
\end{align*}

\underline{Relations \eqref{vava-5}:}
Suppose that $i=k$ or  $i \neq k$ and $\lambda_i^{\bv[k]}=\lambda_i^\bv$. Since $e_j^\bv \tau_i^\bv = e_j^\bv$ and $\tau_k^\bv e_j^\bv \tau_k^\bv \tau_i^\bv = \tau_k^\bv e_j^\bv \tau_k^\bv$ for $j \neq i,k$, we have
\[ e_j^{\bv[k]} \tau_i^\bv = e_j^{\bv[k]} \qquad \text{ for } j \neq i. \] Thus $e_{s,i}^{\bv[k]} \tau_i^\bv = e_{s,i}^\bv$ or $e_{s,i}^{\bv[k]}(1-\tau_i^\bv)=0$, and we have
\begin{align*} s_i^{\bv[k]} s_i^{\bv[k]}& = (\tau_i^\bv + 2(1-\tau_i^\bv) e_{s,i}^{\bv[k]})(\tau_i^\bv + 2(1-\tau_i^\bv) e_{s,i}^{\bv[k]})\\
&= 1+ 2 \tau_i^\bv(1-\tau_i^\bv) e_{s,i}^{\bv[k]} + 2 (1-\tau_i^\bv) e_{s,i}^{\bv[k]}\tau_i^\bv + 4 (1-\tau_i^\bv) e_{s,i}^{\bv[k]}(1-\tau_i^\bv) e_{s,i}^{\bv[k]}\\ 
&= 1+ 2 (\tau_i^\bv -1) e_{s,i}^{\bv[k]} + 2 (1-\tau_i^\bv) e_{s,i}^{\bv[k]}=1.
\end{align*} 
Suppose that $i \neq k$ and $\lambda_i^{\bv[k]} \neq \lambda_i^\bv$. Since
$e_j^{\bv[k]} \tau_k^\bv \tau_i^\bv \tau_k^\bv =  e_j^{\bv[k]}$ for $j \neq i$, the computation is similar to the previous case to obtain $s_i^{\bv[k]} s_i^{\bv[k]}=1$ in this case as well. Furthermore, 
since $e_{\tau,i}^{\bv[k]} s_i^{\bv[k]} = e_{\tau,i}^{\bv[k]}$, 
we get 
\[ \tau_i^{\bv[k]} \tau_i^{\bv[k]} = 
(s_i^{\bv[k]} + 2(1-s_i^{\bv[k]}) e_{s,i}^{\bv[k]})(s_i^{\bv[k]} + 2(1-s_i^{\bv[k]}) e_{\tau,i}^{\bv[k]})=1  . \]

\underline{Relations \eqref{vava-6}:}
Assume $i \neq k$, and suppose that $\lambda_i^{\bv[k]} \neq \lambda_i^\bv$. Then
\begin{align*}
s_i^{\bv[k]} e_i^{\bv[k]} &= (\tau_k^\bv \tau_i^\bv \tau_k^\bv+ 2(1 - \tau_k^\bv \tau_i^\bv \tau_k^\bv) e_{s,i}^{\bv[k]}) e_i^{\bv[k]} \\
&= \tau_k^\bv \tau_i^\bv \tau_k^\bv e_i^{\bv[k]} = \tau_k^\bv \tau_i^\bv \tau_k^\bv \tau_k^\bv e_i^\bv \tau_k^\bv = - \tau_k^\bv e_i^\bv \tau_k^\bv = - e_i^{\bv[k]}.
\end{align*}
The case $\lambda_i^{\bv[k]}= \lambda_i^\bv$ is similar. 
For $i =k$, we obtain
\begin{align*}
s_k^{\bv[k]} e_k^{\bv[k]} &= (\tau_k^\bv + 2(1 - \tau_k^\bv ) e_{s,k}^{\bv[k]}) e_k^{\bv[k]} = \tau_k^\bv e_k^{\bv[k]}\\
& = \tau_k^\bv e_k^\bv ( 1 - \sum_{\ell \neq k , \ \lambda_\ell^{\bv[k]} \neq \lambda_\ell^\bv} e_j^{\bv[k]} ) = - e_k^\bv ( 1 - \sum_{\ell \neq k , \ \lambda_\ell^{\bv[k]} \neq \lambda_\ell^\bv} e_j^{\bv[k]} ) = - e_k^{\bv[k]}.
\end{align*}

For $i \in \mathcal I$, we have
\[ \tau_i^{\bv[k]} e_i^{\bv[k]} = (s_i^{\bv[k]} + 2(1-s_i^{\bv[k]}) e_{\tau,i}^{\bv[k]} ) e_i^{\bv[k]} = s_i^{\bv[k]} e_i^{\bv[k]} = - e_i^{\bv[k]}. \]

\end{proof}

\begin{proof}[Proof of Theorem \ref{con-prec-1}]
The statements (C1) and (C2) are true for $\bw=[]$ from the definitions.
Assume that (C1) and (C2) hold for $\bv$. We will show that they also hold for $\bv[k]$, $k \in \mathcal I$. There are cases (1)-(6) according to the order of $i,j,k$, and each case has several subcases. Since arguments are all similar, we will show details for the cases (1), (3), (4) and (6) and skip some details for the other cases.    

To begin with, let us recall some definitions for ease of reference. 
From the definition of mutation in \eqref{eqn-mmuu}, we have 
\begin{equation} \label{eqn-mmuu-11} b_{ij}^{\bv[k]} = \begin{cases} -b_{ij}^\bv & \text{ if  $i=k$ or $j=k$}, \\ b_{ij}^\bv + \mathrm{sgn}(b_{ik}^\bv) \, b_{ik}^\bv b_{kj}^\bv & \text{ if } b_{ik}^\bv b_{kj}^\bv >0, \\ b_{ij}^\bv & \text{ otherwise}, \end{cases}
\end{equation} and 
rewrite the definition of $c$-vectors as  
\begin{equation} \label{eqn-def-cvec}
c_i^{\bv[k]} = \begin{cases} - c_i^\bv & \text{ if } i=k, \\ c_i^\bv + \mathrm{sgn}(b_{ik}^\bv) b_{ik}^\bv c_k^\bv & \text{ if } b_{ik}^\bv c_k^\bv >0, \\ c_i^\bv & \text{ otherwise}. \end{cases}
\end{equation}

For $i\neq k$, consider the condition
\[ \lambda_i^{\bv} < s_k^{\bv}(\lambda_i^{\bv}) \text{ and } k \prec i,  \text{ or } \lambda_i^{\bv} > s_k^{\bv}(\lambda_i^{\bv})  \text{ and } k \succ i \tag{$*$},\]  and rewrite \eqref{fin-la}, \eqref{eibvk} and \eqref{s_i^}:
\begin{align} \label{fin-la-1} \lambda_i^{\bv[k]} &=  \begin{cases}  \tau_k^{\bv} ( \lambda_i^{\bv} ) & \text{ if ($*$) is true},  \\ \lambda_i^{\bv} & \text{ otherwise}; \end{cases} \\
\label{eibvk-1} 
e_i^{\bv[k]} &=   \begin{cases}  \tau_k^\bv e_i^\bv \tau_k^\bv & \text{ if ($*$) is true}, \\   e_i^\bv & \text{ otherwise}; \end{cases}\\
\label{s_i^-1} 
s_i^{\bv[k]} &= \begin{cases} \tau_k^\bv \tau_i^\bv \tau_k^\bv + 2 ( 1- \tau_k^\bv \tau_i^\bv \tau_k^\bv ) e^{\bv[k]}_{s,i} &  \text{ if ($*$) is true}, \\ \tau_i^\bv+ 2(1-\tau_i^\bv) e^{\bv[k]}_{s,i} &  \text{ otherwise}.  \end{cases} 
\end{align}

In each of the following cases (1)-(6), we will show the statements (C1) and (C2):
\begin{equation*} \lambda_i^\bw =  c_i^\bw \qquad \text{ for all $i \in \mathcal I$}; \tag{\it C1} \end{equation*}
for $i,j \in \mathcal I$,
\begin{align*} 
e_i^\bw(\lambda_j^\bw)& = \delta_{ij} \lambda_j^\bw , & s_i^\bw(\lambda_j^\bw)& = \begin{cases} \lambda_j^\bw +b_{ji}^\bw \lambda_i^\bw & \text{ if } i \prec j, \\ -\lambda_j^\bw & \text{ if } i =j, \\\lambda_j^\bw -b_{ji}^\bw \lambda_i^\bw & \text{ if } i \succ j . \tag{\it C2} \end{cases} 
\end{align*}

\medskip

\noindent \underline{1) Assume that $k \prec i \preceq j$.}
By induction we have
\[ s_k^\bv(\lambda_i^\bv) = \lambda_i^\bv + b_{ik}^\bv \lambda_k^\bv, \qquad s_k^\bv(\lambda_j^\bv) = \lambda_j^\bv + b_{jk}^\bv \lambda_k^\bv.\] 

a) Suppose $b_{ik}^\bv \lambda_k^\bv  =-\lambda_i^\bv +s_k^\bv(\lambda_i^\bv) <0$ and $b_{jk}^\bv\lambda_k^\bv =-\lambda_j^\bv +s_k^\bv(\lambda_j^\bv) <0$. Then from \eqref{eqn-def-cvec}, we  
have \[ c_i^{\bv [k]} =c_i^\bv,\qquad c_j^{\bv [k]} =c_j^\bv,\]
and obtain from \eqref{fin-la-1} 
\[ \lambda_i^{\bv [k]} =\lambda_i^\bv,\qquad \lambda_j^{\bv [k]} =\lambda_j^\bv. \]
By induction, \[ \lambda_i^{\bv[k]} = c_i^{\bv[k]}, \qquad  \lambda_i^{\bv[k]} = c_i^{\bv[k]},\]
which proves (C1) in this case.

From \eqref{eibvk-1},
\[ e_i^{\bv[k]}=e_i^\bv,\qquad  e_j^{\bv[k]}=e_j^\bv, \]
and by induction, 
\begin{align*} e_i^{\bv[k]}(\lambda_j^{\bv[k]})&=e_i^\bv(\lambda_j^\bv) =0,& e_i^{\bv[k]}(\lambda_i^{\bv[k]})&=e_i^\bv(\lambda_i^\bv)=\lambda_i^\bv=\lambda_i^{\bv[k]}, \\
e_j^{\bv[k]}(\lambda_i^{\bv[k]})&=e_j^\bv(\lambda_i^\bv) =0, & e_j^{\bv[k]}(\lambda_j^{\bv[k]})&=e_j^\bv(\lambda_j^\bv)=\lambda_j^\bv=\lambda_j^{\bv[k]}.
\end{align*}
We also have \[s_i^{\bv[k]}=\tau_i^\bv+2(1-\tau_i^\bv) e_{s,i}^{\bv[k]},\qquad s_j^{\bv[k]}=\tau_j^\bv+2(1-\tau_j^\bv) e_{s,j}^{\bv[k]}.\]

From the definitions, $(i,j), (j,i) \not \in \mathcal P_s(\bv, \bv[k]) \cup \mathcal P_\tau(\bv, \bv[k])$, and thus 
\begin{align*}
s_i^{\bv[k]} \lambda_j^{\bv[k]} &= (\tau_i^\bv+2(1-\tau_i^\bv) e_{s,i}^{\bv[k]})\lambda_j^{\bv[k]} 
= \tau_i^\bv \lambda_j^{\bv[k]} = \tau_i^\bv \lambda_j^\bv \\ & = (s_i^\bv + 2(1-s_i^\bv) e_{\tau,i}^\bv ) \lambda_j^\bv = s_i^\bv \lambda_j^\bv \\ & = \begin{cases} \lambda_i^\bv + b_{ji}^\bv \lambda_k^\bv = \lambda_i^{\bv[k]} + b_{ji}^{\bv[k]} \lambda_k^{\bv[k]} & \text{ if } i \neq j,\\ -\lambda_i^\bv = - \lambda_i^{\bv[k]} & \text{ if } i =j . \end{cases} 
\end{align*}
Similarly, we get
\[ s_j^{\bv[k]} \lambda_i^{\bv[k]}= \lambda_i^{\bv[k]} - b_{ij}^{\bv[k]} \lambda_k^{\bv[k]} \quad \text{ for } i \neq j . \]
This proves (C2) in this case.

b) Suppose $b_{ik}^\bv\lambda_k^\bv =-\lambda_i^\bv +s_k^\bv(\lambda_i^\bv) >0$ and $b_{jk}^\bv\lambda_k^\bv =-\lambda_j^\bv +s_k^\bv(\lambda_j^\bv) >0$.
From \eqref{eqn-def-cvec}, we  
have \[ c_i^{\bv [k]} =c_i^\bv+ \mathrm{sgn}(\lambda_k^\bv) b_{ik}^\bv c_k^\bv,\qquad c_j^{\bv [k]} =c_j^\bv+ \mathrm{sgn}(\lambda_k^\bv) b_{jk}^\bv c_k^\bv.\] 
On the other hand, we obtain from \eqref{fin-la-1} 
\[ \lambda_i^{\bv [k]} =\tau_k^\bv(\lambda_i^\bv) = (s_k^\bv + 2( 1- s_k^\bv) e_{\tau,k}^\bv ) (\lambda_i^\bv). \]
If $\lambda_k^\bv <0$ then $(k,i) \in \mathcal P_\tau(\bv, \bv[k])$ and 
\begin{equation} \label{eqn-++}
\lambda_i^{\bv[k]} = (s_k^\bv + 2(1 -s_k^\bv)) (\lambda_i^\bv) = 2 \lambda_i^\bv -s_k^\bv(\lambda_i^\bv) = \lambda_i^\bv - b_{ik}^\bv \lambda_k^\bv = c_i^{\bv[k]} \end{equation}  by induction. 
If $\lambda_k^\bv >0$ then $(k,i) \not \in \mathcal P_\tau(\bv, \bv[k])$ and 
\begin{equation} \label{eqn-+-}
\lambda_i^{\bv[k]} = s_k^\bv \lambda_i^\bv = \lambda_i^\bv + b_{ik}^\bv \lambda_k^\bv = c_i^{\bv[k]}. 
\end{equation}
Similarly, $\lambda_j^{\bv[k]} = c_j^{\bv[k]}$. 
This proves (C1) in this case.

From \eqref{eibvk-1},
\[ e_i^{\bv[k]}=\tau_k^\bv e_i^\bv \tau_k^\bv,\qquad  e_j^{\bv[k]}=\tau_k^\bv e_j^\bv \tau_k^\bv, \]
and by induction, 
\begin{align*} e_i^{\bv[k]}(\lambda_j^{\bv[k]})&=\tau_k^\bv e_i^\bv\tau_k^\bv (\tau_k^\bv \lambda_j^\bv) =\tau_k^\bv e_i^\bv (\lambda_j^\bv)=0, \\  e_i^{\bv[k]}(\lambda_i^{\bv[k]})&=\tau_k^\bv e_i^\bv\tau_k^\bv (\tau_k^\bv \lambda_i^\bv)=\tau_k^\bv e_i^\bv(\lambda_i^\bv)=\tau_k^\bv \lambda_i^\bv=\lambda_i^{\bv[k]}.
\end{align*}
Similarly, $e_j^{\bv[k]}(\lambda_i^{\bv[k]})=0$ and $e_j^{\bv[k]}(\lambda_j^{\bv[k]})=e_j^{\bv[k]}$.

We have
\[s_i^{\bv[k]}=\tau_k^\bv \tau_i^\bv\tau_k^\bv+2(1-\tau_k^\bv\tau_i^\bv\tau_k^\bv) e_{s,i}^{\bv[k]},\qquad s_j^{\bv[k]}=\tau_k^\bv\tau_j^\bv\tau_k^\bv+2(1-\tau_k^\bv\tau_j^\bv\tau_k^\bv) e_{s,j}^{\bv[k]}.\]
From the definitions, $(i,j), (j,i) \not \in \mathcal P_s(\bv, \bv[k]) \cup \mathcal P_\tau(\bv, \bv[k])$, and thus 
\begin{align*}
s_i^{\bv[k]} \lambda_j^{\bv[k]} &= (\tau_k^\bv\tau_i^\bv\tau_k^\bv+2(1-\tau_k^\bv\tau_i^\bv\tau_k^\bv) e_{s,i}^{\bv[k]})\lambda_j^{\bv[k]} 
= \tau_k^\bv\tau_i^\bv \tau_k^\bv\lambda_j^{\bv[k]}= \tau_k^\bv \tau_i^\bv \lambda_j^\bv  \\ & = \tau_k^\bv(s_i^\bv + 2(1-s_i^\bv) e_{\tau,i}^\bv ) \lambda_j^\bv = \tau_k^\bv s_i^\bv \lambda_j^\bv.
\end{align*}
If $i \neq j$ and $\lambda_k^\bv <0$, then we obtain from \eqref{eqn-++}
\begin{align*}
s_i^{\bv[k]} \lambda_j^{\bv[k]} &= \tau_k^\bv s_i^\bv \lambda_j^\bv = \tau_k^\bv ( \lambda_j^\bv + b_{ji}^\bv \lambda_i^\bv) = \tau_k^\bv \lambda_j^\bv + b_{ji}^\bv (s_k^\bv + 2( 1-s_k^\bv) e_{\tau,k}^\bv) \lambda_i^\bv \\ &= \lambda_j^{\bv [k]} + b_{ji}^{\bv [k]}(2-s_k^\bv) \lambda_i^\bv =\lambda_j^{\bv [k]} + b_{ji}^{\bv [k]} \lambda_i^{\bv[k]}.
\end{align*}
If $i \neq j$ and $\lambda_k^\bv >0$, then it follows from \eqref{eqn-+-} that
\begin{align*}
s_i^{\bv[k]} \lambda_j^{\bv[k]} &=  \tau_k^\bv \lambda_j^\bv + b_{ji}^\bv (s_k^\bv + 2( 1-s_k^\bv) e_{\tau,k}^\bv) \lambda_i^\bv \\ &= \lambda_j^{\bv [k]} + b_{ji}^{\bv [k]}s_k^\bv \lambda_i^\bv =\lambda_j^{\bv [k]} + b_{ji}^{\bv [k]} \lambda_i^{\bv[k]}.
\end{align*}
Similarly, we get
\[ s_j^{\bv[k]} \lambda_i^{\bv[k]}= \lambda_i^{\bv[k]} - b_{ij}^{\bv[k]} \lambda_k^{\bv[k]} \quad \text{ for } i \neq j . \]
If $i=j$ then
\begin{align*}
s_i^{\bv[k]} \lambda_i^{\bv[k]} &= (\tau_k^\bv\tau_i^\bv\tau_k^\bv+2(1-\tau_k^\bv\tau_i^\bv\tau_k^\bv) e_{s,i}^{\bv[k]})\lambda_i^{\bv[k]}=\tau_k^\bv \tau_i^\bv \tau_k^\bv \lambda_i^{\bv[k]} = \tau_k^\bv \tau_i^\bv \lambda_i^\bv \\ &= \tau_k^\bv s_i^\bv \lambda_i^\bv = - \tau_k^\bv \lambda_i^\bv = - \lambda_i^{\bv[k]}.\end{align*} 
This proves (C2) in this case.

c) Suppose $b_{ik}^\bv\lambda_k^\bv =-\lambda_i^\bv +s_k^\bv(\lambda_i^\bv) <0$ and $b_{jk}^\bv\lambda_k^\bv =-\lambda_j^\bv +s_k^\bv(\lambda_j^\bv) >0$. 
From \eqref{eqn-def-cvec}, we  
have \[ c_i^{\bv [k]} =c_i^\bv,\qquad c_j^{\bv [k]} =c_j^\bv+ \mathrm{sgn}(\lambda_k^\bv) b_{jk}^\bv c_k^\bv.\] 
On the other hand, we obtain from \eqref{fin-la-1} 
\[ \lambda_i^{\bv [k]} =\lambda_i^\bv, \qquad \lambda_j^{\bv[k]}=\tau_k^\bv(\lambda_j^\bv) = (s_k^\bv + 2( 1- s_k^\bv) e_{\tau,k}^\bv ) (\lambda_j^\bv). \]
Thus $\lambda_i^{\bv[k]} = c_i^{\bv[k]}$ by induction, and using the same argument as in (b), we also see that $\lambda_j^{\bv[k]} = c_j^{\bv[k]}$. 
Therefore (C1) is true in this case.

From \eqref{eibvk-1},
\[ e_i^{\bv[k]}= e_i^\bv ,\qquad  e_j^{\bv[k]}=\tau_k^\bv e_j^\bv \tau_k^\bv, \]
and it follows from similar computations to those in (a) and (b) that 
\begin{align*} e_i^{\bv[k]}(\lambda_j^{\bv[k]})&=0,&   e_i^{\bv[k]}(\lambda_i^{\bv[k]})&=\lambda_i^{\bv[k]},\\ e_j^{\bv[k]}(\lambda_i^{\bv[k]})&=0,&   e_j^{\bv[k]}(\lambda_j^{\bv[k]})&=\lambda_j^{\bv[k]}.
\end{align*}

We have
\[s_i^{\bv[k]}=\tau_i^\bv+2(1-\tau_i^\bv) e_{s,i}^{\bv[k]},\qquad s_j^{\bv[k]}=\tau_k^\bv\tau_j^\bv\tau_k^\bv+2(1-\tau_k^\bv\tau_j^\bv\tau_k^\bv) e_{s,j}^{\bv[k]}.\]
From the definitions, $(i,j) \in \mathcal P_s(\bv, \bv[k]) \cap \mathcal P_\tau(\bv, \bv[k])$, and thus 
\begin{align*}
s_i^{\bv[k]} \lambda_j^{\bv[k]} &= (\tau_i^\bv+2(1-\tau_i^\bv) e_{s,i}^{\bv[k]})\lambda_j^{\bv[k]} 
= \tau_i^\bv \lambda_j^{\bv[k]}+2(1-\tau_i^\bv) \lambda_j^{\bv[k]}= 2 \lambda_j^{\bv[k]} - \tau_i^\bv \lambda_j^{\bv[k]}  \\ & =2 \lambda_j^{\bv[k]} - (s_i^\bv + 2(1-s_i^\bv) e_{\tau,i}^\bv ) \lambda_j^{\bv[k]} .
\end{align*}
If $i \neq j$ and $\lambda_k^\bv <0$, then $(k,i) \not \in \mathcal P_\tau(\bv, \bv[k])$, $(k,j) \in \mathcal P_\tau(\bv, \bv[k])$, and thus $\lambda_j^{\bv[k]}=\tau_k^\bv (\lambda_j^\bv) = \lambda_j^\bv -b_{jk}^\bv \lambda_k^\bv$ and 
by \eqref{eqn-mmuu-11}
\begin{align*}
s_i^{\bv[k]} \lambda_j^{\bv[k]} &= 2 \lambda_j^{\bv[k]} - (s_i^\bv + 2(1-s_i^\bv) e_{\tau,i}^\bv )(\lambda_j^\bv -b_{jk}^\bv \lambda_k^\bv) \\
&=2 \lambda_j^{\bv[k]} -(s_i^\bv \lambda_j^\bv - b_{jk}^\bv ( \lambda_k^\bv - b_{ki}^\bv \lambda_i^\bv) + 2 (1 -s_i^\bv) \lambda_j^\bv) \\
&= 2 \lambda_j^{\bv[k]} - (2\lambda_j^\bv -s_i^\bv \lambda_j^\bv - b_{jk}^\bv \lambda_k^\bv + b_{jk}^\bv b_{ki}^\bv \lambda_i^\bv ) \\
&= 2 \lambda_j^{\bv[k]} - (\lambda_j^\bv - b_{ji}^\bv \lambda_i^\bv - b_{jk}^\bv \lambda_k^\bv + b_{jk}^\bv b_{ki}^\bv \lambda_i^\bv ) \\
&= \lambda_j^{\bv[k]} + (b_{ji}^\bv - b_{jk}^\bv b_{ki}^\bv) \lambda_i^{\bv[k]} = \lambda_j^{\bv[k]} + b_{ji}^{\bv[k]} \lambda_i^{\bv[k]}. 
\end{align*}
If $i \neq j$ and $\lambda_k^\bv >0$, then $(k,i), (k,j) \not \in \mathcal P_\tau(\bv, \bv[k])$, and thus $\lambda_j^{\bv[k]}=\tau_k^\bv (\lambda_j^\bv) = \lambda_j^\bv +b_{jk}^\bv \lambda_k^\bv$ and by \eqref{eqn-mmuu-11}  
\begin{align*}
s_i^{\bv[k]} \lambda_j^{\bv[k]} &= 2 \lambda_j^{\bv[k]} - (s_i^\bv + 2(1-s_i^\bv) e_{\tau,i}^\bv )(\lambda_j^\bv +b_{jk}^\bv \lambda_k^\bv) \\
&= \lambda_j^{\bv[k]} + (b_{ji}^\bv + b_{jk}^\bv b_{ki}^\bv) \lambda_i^{\bv[k]} = \lambda_j^{\bv[k]} + b_{ji}^{\bv[k]} \lambda_i^{\bv[k]}. 
\end{align*}
Similarly, we get
\[ s_j^{\bv[k]} \lambda_i^{\bv[k]}= \lambda_i^{\bv[k]} - b_{ij}^{\bv[k]} \lambda_k^{\bv[k]} \quad \text{ for } i \neq j \quad \text{ and } \quad  s_i^{\bv[k]} \lambda_i^{\bv[k]}= - \lambda_i^{\bv[k]}.  \]
This proves (C2) in this case.

d) Suppose $b_{ik}^\bv\lambda_k^\bv =-\lambda_i^\bv +s_k^\bv(\lambda_i^\bv) >0$ and $b_{jk}^\bv\lambda_k^\bv =-\lambda_j^\bv +s_k^\bv(\lambda_j^\bv) <0$.
This case is similar to case (c) right above.

\medskip

\noindent \underline{2) Assume that $i \preceq j \prec k$.}
Since this case is similar to case (1), we omit the details.

\medskip

\noindent \underline{3) Assume that $i \prec k \prec j$.}
By induction we have
\[ s_k^\bv(\lambda_i^\bv) = \lambda_i^\bv - b_{ik}^\bv \lambda_k^\bv, \qquad s_k^\bv(\lambda_j^\bv) = \lambda_j^\bv + b_{jk}^\bv \lambda_k^\bv.\] 

a) Suppose $b_{ik}^\bv\lambda_k^\bv =\lambda_i^\bv -s_k^\bv(\lambda_i^\bv) <0$ and $b_{jk}^\bv\lambda_k^\bv =-\lambda_j^\bv +s_k^\bv(\lambda_j^\bv) <0$.
From \eqref{eqn-def-cvec}, we  
have \[ c_i^{\bv [k]} =c_i^\bv,\qquad c_j^{\bv [k]} =c_j^\bv.\] 
It follows from  \eqref{fin-la-1} that 
\[ \lambda_i^{\bv [k]} =\lambda_i^\bv, \qquad \lambda_j^{\bv[k]}=\lambda_j^\bv. \]
Thus $\lambda_i^{\bv[k]} = c_i^{\bv[k]}$ and $\lambda_j^{\bv[k]} = c_j^{\bv[k]}$ by induction. Thus (C1) is true in this case.

From \eqref{eibvk-1},
\[ e_i^{\bv[k]}= e_i^\bv ,\qquad  e_j^{\bv[k]}= e_j^\bv, \]
and it follows from induction that 
\begin{align*} e_i^{\bv[k]}(\lambda_j^{\bv[k]})&=0,&   e_i^{\bv[k]}(\lambda_i^{\bv[k]})&=\lambda_i^{\bv[k]},\\ e_j^{\bv[k]}(\lambda_i^{\bv[k]})&=0,&   e_j^{\bv[k]}(\lambda_j^{\bv[k]})&=\lambda_j^{\bv[k]}.
\end{align*}

We have
\[s_i^{\bv[k]}=\tau_i^\bv+2(1-\tau_i^\bv) e_{s,i}^{\bv[k]},\qquad s_j^{\bv[k]}=\tau_j^\bv+2(1-\tau_j^\bv) e_{s,j}^{\bv[k]}.\]
Clearly, $(i,j), (j,i) \not\in \mathcal P_s(\bv, \bv[k]) \cup \mathcal P_\tau(\bv, \bv[k])$, and thus 
\begin{align*}
s_i^{\bv[k]} \lambda_j^{\bv[k]} &= (\tau_i^\bv+2(1-\tau_i^\bv) e_{s,i}^{\bv[k]})\lambda_j^{\bv[k]} 
= \tau_i^\bv \lambda_j^{\bv[k]}= (s_i^\bv + 2(1-s_i^\bv) e_{\tau,i}^\bv ) \lambda_j^\bv \\ &= s_i^\bv \lambda_j^\bv = \lambda_j^\bv + b_{ji}^\bv \lambda_i^\bv\\ & = \begin{cases} \lambda_i^\bv + b_{ji}^\bv \lambda_k^\bv = \lambda_i^{\bv[k]} + b_{ji}^{\bv[k]} \lambda_k^{\bv[k]} & \text{ if } i \neq j,\\ -\lambda_i^\bv = - \lambda_i^{\bv[k]} & \text{ if } i =j . \end{cases}
\end{align*}
Similarly, we get
\[ s_j^{\bv[k]} \lambda_i^{\bv[k]}= \lambda_i^{\bv[k]} - b_{ij}^{\bv[k]} \lambda_k^{\bv[k]} \quad \text{ for } i \neq j \quad \text{ and } \quad  s_i^{\bv[k]} \lambda_i^{\bv[k]}= - \lambda_i^{\bv[k]}.  \]
This proves (C2) in this case.

b) Suppose $b_{ik}^\bv\lambda_k^\bv =\lambda_i^\bv -s_k^\bv(\lambda_i^\bv) >0$ and $b_{jk}^\bv\lambda_k^\bv =-\lambda_j^\bv +s_k^\bv(\lambda_j^\bv) >0$.
From \eqref{eqn-def-cvec}, we  
have \[ c_i^{\bv [k]} =c_i^\bv+ \mathrm{sgn}(\lambda_k^\bv) b_{ik}^\bv c_k^\bv,\qquad c_j^{\bv [k]} =c_j^\bv+ \mathrm{sgn}(\lambda_k^\bv) b_{jk}^\bv c_k^\bv.\] 
We obtain from \eqref{fin-la-1} 
\[ \lambda_i^{\bv [k]} =\tau_k^\bv(\lambda_i^\bv) = (s_k^\bv + 2( 1- s_k^\bv) e_{\tau,k}^\bv ) (\lambda_i^\bv). \]
If $\lambda_k^\bv >0$ then $(k,i) \in \mathcal P_\tau(\bv, \bv[k])$ and 
\begin{equation} \label{eqn-+++}
\lambda_i^{\bv[k]} = (s_k^\bv + 2(1 -s_k^\bv)) (\lambda_i^\bv) = 2 \lambda_i^\bv -s_k^\bv(\lambda_i^\bv) = \lambda_i^\bv + b_{ik}^\bv \lambda_k^\bv = c_i^{\bv[k]} \end{equation}  by induction. 
If $\lambda_k^\bv <0$ then $(k,i) \not \in \mathcal P_\tau(\bv, \bv[k])$ and 
\begin{equation} \label{eqn-+--}
\lambda_i^{\bv[k]} = s_k^\bv \lambda_i^\bv = \lambda_i^\bv - b_{ik}^\bv \lambda_k^\bv = c_i^{\bv[k]}. 
\end{equation}
Similarly, $\lambda_j^{\bv[k]} = c_j^{\bv[k]}$. 
This proves (C1) in this case.

From \eqref{eibvk-1},
\[ e_i^{\bv[k]}=\tau_k^\bv e_i^\bv \tau_k^\bv,\qquad  e_j^{\bv[k]}=\tau_k^\bv e_j^\bv \tau_k^\bv, \]
and it follows from induction that 
\begin{align*} e_i^{\bv[k]}(\lambda_j^{\bv[k]})&=0,&   e_i^{\bv[k]}(\lambda_i^{\bv[k]})&=\lambda_i^{\bv[k]},\\ e_j^{\bv[k]}(\lambda_i^{\bv[k]})&=0,&   e_j^{\bv[k]}(\lambda_j^{\bv[k]})&=\lambda_j^{\bv[k]}.
\end{align*}

We have
\[s_i^{\bv[k]}=\tau_k^\bv \tau_i^\bv\tau_k^\bv+2(1-\tau_k^\bv\tau_i^\bv\tau_k^\bv) e_{s,i}^{\bv[k]},\qquad s_j^{\bv[k]}=\tau_k^\bv\tau_j^\bv\tau_k^\bv+2(1-\tau_k^\bv\tau_j^\bv\tau_k^\bv) e_{s,j}^{\bv[k]}.\]
Clearly, $(i,j), (j,i) \not \in \mathcal P_s(\bv, \bv[k]) \cup \mathcal P_\tau(\bv, \bv[k])$, and as in (1)-(b),  
\begin{align*}
s_i^{\bv[k]} \lambda_j^{\bv[k]} &= \tau_k^\bv s_i^\bv \lambda_j^\bv.
\end{align*}
If $\lambda_k^\bv >0$, then we obtain from \eqref{eqn-+++}
\begin{align*}
s_i^{\bv[k]} \lambda_j^{\bv[k]} &= \tau_k^\bv s_i^\bv \lambda_j^\bv = \tau_k^\bv ( \lambda_j^\bv + b_{ji}^\bv \lambda_i^\bv) = \tau_k^\bv \lambda_j^\bv + b_{ji}^\bv (s_k^\bv + 2( 1-s_k^\bv) e_{\tau,k}^\bv) \lambda_i^\bv \\ &= \lambda_j^{\bv [k]} + b_{ji}^{\bv [k]}(2-s_k^\bv) \lambda_i^\bv =\lambda_j^{\bv [k]} + b_{ji}^{\bv [k]} \lambda_i^{\bv[k]}.
\end{align*}
If $\lambda_k^\bv <0$, then it follows from \eqref{eqn-+--} that
\begin{align*}
s_i^{\bv[k]} \lambda_j^{\bv[k]} &=  \tau_k^\bv \lambda_j^\bv + b_{ji}^\bv (s_k^\bv + 2( 1-s_k^\bv) e_{\tau,k}^\bv) \lambda_i^\bv \\ &= \lambda_j^{\bv [k]} + b_{ji}^{\bv [k]}s_k^\bv \lambda_i^\bv =\lambda_j^{\bv [k]} + b_{ji}^{\bv [k]} \lambda_i^{\bv[k]}.
\end{align*}
Similarly, we get
\[ s_j^{\bv[k]} \lambda_i^{\bv[k]}= \lambda_i^{\bv[k]} - b_{ij}^{\bv[k]} \lambda_k^{\bv[k]}. \]
This proves (C2) in this case.

c) Suppose $b_{ik}^\bv\lambda_k^\bv =\lambda_i^\bv -s_k^\bv(\lambda_i^\bv) <0$ and $b_{jk}^\bv\lambda_k^\bv =-\lambda_j^\bv +s_k^\bv(\lambda_j^\bv) >0$. 
From \eqref{eqn-def-cvec}, we  
have \[ c_i^{\bv [k]} =c_i^\bv,\qquad c_j^{\bv [k]} =c_j^\bv+ \mathrm{sgn}(\lambda_k^\bv) b_{jk}^\bv c_k^\bv.\] 
On the other hand, we obtain from \eqref{fin-la-1} 
\[ \lambda_i^{\bv [k]} =\lambda_i^\bv, \qquad \lambda_j^{\bv[k]}=\tau_k^\bv(\lambda_j^\bv) = (s_k^\bv + 2( 1- s_k^\bv) e_{\tau,k}^\bv ) (\lambda_j^\bv). \]
Thus $\lambda_i^{\bv[k]} = c_i^{\bv[k]}$ by induction, and using the same argument as in (b), we also see that $\lambda_j^{\bv[k]} = c_j^{\bv[k]}$. 
Therefore (C1) is true in this case.

From \eqref{eibvk-1},
\[ e_i^{\bv[k]}= e_i^\bv ,\qquad  e_j^{\bv[k]}=\tau_k^\bv e_j^\bv \tau_k^\bv, \]
and it follows from induction that 
\begin{align*} e_i^{\bv[k]}(\lambda_j^{\bv[k]})&=0,&   e_i^{\bv[k]}(\lambda_i^{\bv[k]})&=\lambda_i^{\bv[k]},\\ e_j^{\bv[k]}(\lambda_i^{\bv[k]})&=0,&   e_j^{\bv[k]}(\lambda_j^{\bv[k]})&=\lambda_j^{\bv[k]}.
\end{align*}

We have
\[s_i^{\bv[k]}=\tau_i^\bv+2(1-\tau_i^\bv) e_{s,i}^{\bv[k]},\qquad s_j^{\bv[k]}=\tau_k^\bv\tau_j^\bv\tau_k^\bv+2(1-\tau_k^\bv\tau_j^\bv\tau_k^\bv) e_{s,j}^{\bv[k]}.\]
From the definitions, $(i,j), (j,i) \not \in \mathcal P_s(\bv, \bv[k])$, and thus 
\begin{align*}
s_i^{\bv[k]} \lambda_j^{\bv[k]} &= (\tau_i^\bv+2(1-\tau_i^\bv) e_{s,i}^{\bv[k]})\lambda_j^{\bv[k]} 
= \tau_i^\bv \lambda_j^{\bv[k]}.
\end{align*}
If $\lambda_k^\bv <0$, then $(k,i) \not\in \mathcal P_\tau(\bv, \bv[k])$,
 $(k,j) \in \mathcal P_\tau(\bv, \bv[k])$, and thus $\lambda_j^{\bv[k]}=\tau_k^\bv (\lambda_j^\bv) = \lambda_j^\bv -b_{jk}^\bv \lambda_k^\bv$ and by \eqref{eqn-mmuu-11}
\begin{align*}
s_i^{\bv[k]} \lambda_j^{\bv[k]} &= \tau_i^\bv \lambda_j^{\bv[k]} = \tau_i^\bv ( \lambda_j^\bv - b_{jk}^\bv \lambda_k^\bv)  =(s_i^\bv + 2(1-s_i^\bv) e_{\tau,i}^\bv )(\lambda_j^\bv -b_{jk}^\bv \lambda_k^\bv) \\
&=s_i^\bv \lambda_j^\bv - b_{jk}^\bv s_i^\bv \lambda_k^\bv  
= \lambda_j^\bv + b_{ji}^\bv \lambda_i^\bv - b_{jk}^\bv ( \lambda_k^\bv + b_{ki}^\bv \lambda_i^\bv)\\ &= \lambda_j^\bv - b_{jk}^\bv + (b_{ji}^\bv - b_{jk}^\bv b_{ki}^\bv) \lambda_i^\bv  = \lambda_j^{\bv[k]} + b_{ji}^{\bv[k]} \lambda_i^{\bv[k]}. 
\end{align*}
If $\lambda_k^\bv >0$, then $(k,i), (k,j) \not \in \mathcal P_\tau(\bv, \bv[k])$ and thus $\lambda_j^{\bv[k]}=\tau_k^\bv (\lambda_j^\bv) = \lambda_j^\bv +b_{jk}^\bv \lambda_k^\bv$ and by \eqref{eqn-mmuu-11}  
\begin{align*}
s_i^{\bv[k]} \lambda_j^{\bv[k]} &= \tau_i^\bv \lambda_j^{\bv[k]} = (s_i^\bv + 2(1-s_i^\bv) e_{\tau,i}^\bv )(\lambda_j^\bv +b_{jk}^\bv \lambda_k^\bv) \\ 
&=s_i^\bv \lambda_j^\bv + b_{jk}^\bv s_i^\bv\lambda_k^\bv  
= \lambda_j^\bv + b_{ji}^\bv \lambda_i^\bv + b_{jk}^\bv ( \lambda_k^\bv + b_{ki}^\bv \lambda_i^\bv) \\ &= \lambda_j^{\bv[k]} + (b_{ji}^\bv + b_{jk}^\bv b_{ki}^\bv) \lambda_i^{\bv[k]} = \lambda_j^{\bv[k]} + b_{ji}^{\bv[k]} \lambda_i^{\bv[k]}. 
\end{align*}
Similarly, we get
\[ s_j^{\bv[k]} \lambda_i^{\bv[k]}= \lambda_i^{\bv[k]} - b_{ij}^{\bv[k]} \lambda_k^{\bv[k]} \quad \text{ for } i \neq j \quad \text{ and } \quad  s_i^{\bv[k]} \lambda_i^{\bv[k]}= - \lambda_i^{\bv[k]}.  \]
This proves (C2) in this case.

d) Suppose $b_{ik}^\bv\lambda_k^\bv =\lambda_i^\bv -s_k^\bv(\lambda_i^\bv) >0$ and $b_{jk}^\bv\lambda_k^\bv =-\lambda_j^\bv +s_k^\bv(\lambda_j^\bv) <0$.
This case is similar to (c) and we omit the details.
\medskip

\noindent \underline{4) Assume that $i \prec k =j$.}
By induction we have
\[ s_k^\bv(\lambda_i^\bv) = \lambda_i^\bv - b_{ik}^\bv \lambda_k^\bv, \qquad s_k^\bv(\lambda_k^\bv) =-\lambda_k^\bv .\] 

a) Suppose $b_{ik}^\bv\lambda_k^\bv =\lambda_i^\bv -s_k^\bv(\lambda_i^\bv) <0$.
From \eqref{eqn-def-cvec}, we  
have \[ c_i^{\bv [k]} =c_i^\bv,\qquad c_k^{\bv [k]} =-c_k^\bv.\] 
Since $(k,k) \not \in \mathcal P_{\tau}(\bv, \bv[k])$, we obtain from \eqref{fin-la} and induction 
\begin{align} \lambda_i^{\bv [k]} &=\lambda_i^\bv, \nonumber \\
\lambda_k^{\bv[k]}&=\tau_k^\bv (\lambda_k^\bv) =(s_k^\bv + 2(1-s_k^\bv)e_{\tau,k}^\bv) \lambda_k^\bv = s_k^\bv \lambda_k^\bv =-\lambda_k^\bv. \label{eqn-llb-} 
\end{align}
Thus $\lambda_i^{\bv[k]} = c_i^{\bv[k]}$ and $\lambda_k^{\bv[k]} = c_k^{\bv[k]}$ by induction, and  (C1) is true in this case.

From \eqref{eibvk-1} and \eqref{vava-1},
\[ e_i^{\bv[k]}= e_i^\bv ,\qquad  e_k^{\bv[k]}= 1- \sum_{\ell \neq k} e_\ell^{\bv[k]}, \]
and it follows from induction that 
\begin{align*} e_i^{\bv[k]}(\lambda_k^{\bv[k]})&=e_i^\bv( - \lambda_k^\bv)=0,\qquad   e_i^{\bv[k]}(\lambda_i^{\bv[k]})=e_i^\bv \lambda_i^\bv=\lambda_i^{\bv[k]},\\ e_k^{\bv[k]}(\lambda_i^{\bv[k]})&=(1- \sum_{\ell \neq k} e_\ell^{\bv[k]}) \lambda_i^{\bv[k]} = \lambda_i^{\bv[k]}- \lambda_i^{\bv[k]} =0.&  \end{align*}

We have
\[s_i^{\bv[k]}=\tau_i^\bv+2(1-\tau_i^\bv) e_{s,i}^{\bv[k]},\qquad s_k^{\bv[k]}=\tau_k^\bv+2(1-\tau_k^\bv) e_{s,k}^{\bv[k]}.\]
We see that  $(k,i) \not\in \mathcal P_s(\bv, \bv[k]) \cup \mathcal P_\tau(\bv, \bv[k])$, and thus 
\begin{align*}
s_i^{\bv[k]} \lambda_k^{\bv[k]} &= (\tau_i^\bv+2(1-\tau_i^\bv) e_{s,i}^{\bv[k]})\lambda_k^{\bv[k]} 
= \tau_i^\bv \lambda_k^{\bv[k]}= -(s_i^\bv + 2(1-s_i^\bv) e_{\tau,i}^\bv ) \lambda_k^\bv \\ &= -s_i^\bv \lambda_k^\bv = -\lambda_k^\bv - b_{ki}^\bv \lambda_i^\bv = \lambda_k^{\bv[k]} +b_{ki}^{\bv[k]} \lambda_i^{\bv[k]}. 
\end{align*}
Similarly, we get
\[ s_k^{\bv[k]} \lambda_i^{\bv[k]}= \lambda_i^{\bv[k]} - b_{ik}^{\bv[k]} \lambda_k^{\bv[k]}.  \]
This proves (C2) in this case.

b) Suppose $b_{ik}^\bv\lambda_k^\bv = \lambda_i^\bv -s_k^\bv(\lambda_i^\bv) >0$.
From \eqref{eqn-def-cvec}, we  
have \[ c_i^{\bv [k]} =c_i^\bv+ \mathrm{sgn}(\lambda_k^\bv) b_{ik}^\bv c_k^\bv, \qquad c_k^{\bv [k]} =-c_k^\bv.\] 
On the other hand, we obtain from \eqref{fin-la} 
\[ \lambda_i^{\bv [k]} =\tau_k^\bv(\lambda_i^\bv) = (s_k^\bv + 2( 1- s_k^\bv) e_{\tau,k}^\bv ) (\lambda_i^\bv), \qquad \lambda_k^{\bv[k]}=-\lambda_k^\bv. \]
If $\lambda_k^\bv <0$ then $(k,i) \not\in \mathcal P_\tau(\bv, \bv[k])$ and $\lambda_i^{\bv[k]} = s_k^\bv \lambda_i^\bv = \lambda_i^\bv-b_{ik}^\bv \lambda_k^\bv$;
if $\lambda_k^\bv >0$ then $(k,i) \in \mathcal P_\tau(\bv, \bv[k])$ and $\lambda_i^{\bv[k]} = (2-s_k^\bv)\lambda_i^\bv = \lambda_i^\bv+b_{ik}^\bv \lambda_k^\bv$.  Thus $\lambda_i^{\bv[k]} = c_i^{\bv[k]}$ and $\lambda_k^{\bv[k]} = c_k^{\bv[k]}$ by induction, and (C1) is true in this case.

From \eqref{eibvk-1}, \eqref{vava-1} and \eqref{vava-4},
\[ e_i^{\bv[k]}= \tau_k^\bv e_i^\bv\tau_k^\bv=\tau_k^\bv e_i^\bv ,\qquad  e_k^{\bv[k]}= 1- \sum_{\ell \neq k} e_\ell^{\bv[k]}, \]
and it follows from induction that 
\begin{align*} e_i^{\bv[k]}(\lambda_k^{\bv[k]})&=\tau_k^\bv e_i^\bv( - \lambda_k^\bv)=0,\qquad   e_i^{\bv[k]}(\lambda_i^{\bv[k]})=\tau_k^\bv e_i^\bv \tau_k^\bv \tau_k^\bv \lambda_i^\bv=\tau_k^\bv \lambda_i^\bv = \lambda_i^{\bv[k]},\\ e_k^{\bv[k]}(\lambda_i^{\bv[k]})&=(1- \sum_{\ell \neq k} e_\ell^{\bv[k]}) \lambda_i^{\bv[k]} = \lambda_i^{\bv[k]}- \lambda_i^{\bv[k]} =0.&  \end{align*}

We have
\[s_i^{\bv[k]}=\tau_k^\bv \tau_i^\bv \tau_k^\bv+2(1-\tau_k^\bv\tau_i^\bv\tau_k^\bv) e_{s,i}^{\bv[k]},\qquad s_k^{\bv[k]}=\tau_k^\bv+2(1-\tau_k^\bv) e_{s,k}^{\bv[k]}.\]
If $\lambda_k^\bv<0$, then $(k,i) \not\in \mathcal P_\tau(\bv, \bv[k])$ and $(k,i) \in \mathcal P_s(\bv, \bv[k])$, and thus \begin{align*}
s_i^{\bv[k]} \lambda_k^{\bv[k]} &= (\tau_k^\bv\tau_i^\bv\tau_k^\bv+2(1-\tau_k^\bv\tau_i^\bv\tau_k^\bv) e_{s,i}^{\bv[k]})\lambda_k^{\bv[k]}= (2 - \tau_k^\bv \tau_i^\bv \tau_k^\bv)(-\lambda_k^\bv)\\& = -2 \lambda_k^\bv - \tau_k^\bv \tau_i^\bv \lambda_k^\bv = -2 \lambda_k^\bv -\tau_k(s_i^\bv + 2 (1 -s_i^\bv)e_{\tau,i}^\bv) \lambda_k^\bv \\ & = -2 \lambda_k^\bv -\tau_k^\bv s_i^\bv \lambda_k^\bv = -2 \lambda_k^\bv - \tau_k^\bv(\lambda_k^\bv+ b_{ki}^\bv \lambda_i^\bv) \\ &= -2 \lambda_k^\bv +\lambda_k^\bv - b_{ki}^\bv \tau_k^\bv \lambda_i^\bv = \lambda_k^{\bv[k]} + b_{ki}^{\bv[k]} \lambda_i^{\bv[k]},
\end{align*}
and since
$\lambda_i^{\bv[k]}=\tau_k^\bv \lambda_i^\bv = s_k^\bv \lambda_i^\bv = \lambda_i^\bv -b_{ik}^\bv \lambda_k^\bv,$ 
we have
\begin{align*}
s_k^{\bv[k]} \lambda_i^{\bv[k]} &= (\tau_k^\bv + 2( 1-\tau_k^\bv) e_{s,i}^\bv) \lambda_i^{\bv[k]}  = (2-\tau_k^\bv) \tau_k^\bv \lambda_i^\bv = 2 \tau_k^\bv \lambda_i^\bv - \lambda_i^\bv \\ & = 2 (\lambda_i^\bv -b_{ik}^\bv \lambda_k^\bv) - \lambda_i^\bv = (\lambda_i^\bv - b_{ik}^\bv \lambda_k^\bv) - b_{ik}^\bv \lambda_k^\bv = \lambda_i^{\bv[k]} -b_{ik}^{\bv[k]} \lambda_k^{\bv[k]}.
\end{align*}
If $\lambda_k^\bv>0$, then $(k,i) \in \mathcal P_\tau(\bv, \bv[k])$ and $(k,i) \not \in \mathcal P_s(\bv, \bv[k])$, and the computations are similar to the case right above.
This proves (C2) in this case.

\medskip

\noindent \underline{5) Assume that $i=k \prec j$.}
Since this case is similar to case (4), we omit the details. 

\medskip

\noindent \underline{6) Assume that $i =j=k$.}
From \eqref{eqn-def-cvec}, we  
have $c_k^{\bv [k]} =-c_k^\bv$. As seen in \eqref{eqn-llb-}, we have
$\lambda_k^{\bv[k]}=-\lambda_k^\bv$. Thus by induction $c_k^{\bv[k]}=\lambda_k^{\bv[k]}$, and (C1) holds.
In cases (4) and (5), it is proven that $e_\ell^{\bv[k]} \lambda_k^{\bv[k]} = 0$ for $\ell \neq k$. Thus using \eqref{vava-1}, we have
\[ e_k^{\bv[k]} \lambda_k^{\bv[k]}= (1 - \sum_{\ell \neq k} e_\ell^{\bv[k]}) \lambda_k^{\bv[k]} = \lambda_k^{\bv[k]}. \]
Finally, since $(k,k) \not \in \mathcal P_s(\bv, \bv[k])$, we see that 
\[ s_k^{\bv[k]} \lambda_k^{\bv[k]} = (\tau_k^\bv+2(1-\tau_k^\bv) e_{s,k}^{\bv[k]} ) \lambda_k^{\bv[k]} = \tau_k^\bv \lambda_k^{\bv[k]} = \tau_k^\bv \tau_k^\bv \lambda_k^\bv = \lambda_k^\bv = - \lambda_k^{\bv[k]}, \] where we use \eqref{vava-5}. 
This proves (C2) in this case, and a proof of Theorem \ref{con-prec-1} has been completed.  

\end{proof}

\end{document}